\documentclass{amsart}
\usepackage{amsmath} 
\usepackage{stmaryrd}
\usepackage{tikz}
\usetikzlibrary{decorations.pathreplacing}
\usetikzlibrary{positioning}
\usepackage{amsthm,amssymb}
\usepackage[unicode]{hyperref}
\hypersetup{
	colorlinks=true,
	linkcolor=black,
	citecolor=black,
	urlcolor=blue
}

\usepackage{xcolor}
\usepackage{mathtools}
\usepackage{calc}
\usepackage{graphicx}
\usepackage{bm}
\usepackage{pifont}
\usepackage{enumitem}
\usepackage{xspace}

\newtheorem{thm}{Theorem}[section]
\newtheorem{lem}[thm]{Lemma}
\newtheorem{qst}[thm]{Question}
\newtheorem{cor}[thm]{Corollary}

\newtheorem{prop}[thm]{Proposition}

\newtheorem{claim}[thm]{Claim}

\theoremstyle{definition}
\newtheorem{defn}[thm]{Definition}
\newtheorem{fact}[thm]{Fact}

\newtheorem{rem}[thm]{Remark}

\numberwithin{equation}{section}

\newcommand{\lb}{\left(}
\newcommand{\rb}{\right)}
\newcommand{\la}{\langle}
\newcommand{\ra}{\rangle}

\newcommand{\lrtar}{\leftrightarrow}

\let\standardwidetilde\widetilde
\let\standardwidehat\widehat
\usepackage{yhmath}

\let\yhwidehat\widehat
\let\widetilde\standardwidetilde
\let\widehat\standardwidehat

\newcommand{\what}[1]{\yhwidehat{#1}}
\newcommand{\wbar}[1]{\overline{#1}}
\newcommand{\uphar}{\upharpoonright}
\newcommand{\idl}[1]{{\mathcal #1}}
\newcommand{\mdl}[1]{{\mathfrak #1}}

\newcommand{\bfsgm}{{\boldsymbol \Sigma}}

\newcommand{\infsub}[1]{{[#1]^\omega}}
\newcommand{\finsub}[1]{{[#1]^{<\omega}}}

\newcommand{\finfunc}[1]{{#1^{<\omega}}}
\newcommand{\fintree}{{\omega^{<\omega}}}
\newcommand{\bairesp}{{\omega^\omega}}

\newcommand{\ramseysp}{{[\omega]^\omega}}

\newcommand{\prst}{{\mathcal P}}

\newcommand{\dom}{\operatorname{dom}}
\newcommand{\ran}{{\rm ran}}

\newcommand{\compl}{{\complement}}

\newcommand{\chaxiom}{{\sf CH}}

\newcommand{\encfunc}{{\mathrm{e}}}

\newcommand{\finalpha}[1]{\mathrm{Fin}^{#1}}

\newcommand{\evendiff}{\mathcal{ED}}
\newcommand{\Eq}[1]{\operatorname{Eq}_{#1}}
\newcommand{\gr}{\operatorname{gr}}

\newcommand{\Ncal}{\mathcal{N}}
\newcommand{\U}{\mathcal{U}}

\newcommand{\VD}{\ensuremath{\operatorname{VD}}\xspace}
\newcommand{\tagg}{\operatorname{t}}
\newcommand{\base}{\operatorname{b}}
\newcommand{\real}{\operatorname{r}}
\newcommand{\degree}{\operatorname{d}}
\newcommand{\Breq}{\mathfrak{B}}
\newcommand{\rec}{\operatorname{rec}}
\newcommand{\good}{\operatorname{good}}
\newcommand{\prodspc}[1]{\idl{N}_{#1}}
\newcommand{\trace}{\operatorname{tr}}
\newcommand{\FP}{\ensuremath{\operatorname{FP}}\xspace}
\newcommand{\FSP}{\ensuremath{\operatorname{FSP}}\xspace}
\newcommand{\omegaCP}{\ensuremath{\omega\text{-}\mathrm{CP}}\xspace}
\newcommand{\own}{\operatorname{own}}

\newenvironment{claimproof}[1]{
	\renewcommand{\qedsymbol}{{\tiny Claim #1} \(\square\)}
	\begin{proof}[Proof of Claim #1]
	}{
	\end{proof}
}

\title{Van Douwen Families, Productivity and Ultrafilter Maximality}

\author{Jialiang He}
\address{School of Mathematics, Sichuan University, Chengdu 610064, China}
\email{jialianghe@scu.edu.cn}

\author{Jintao Luo}
\address{Independent researcher}
\curraddr{Universit\'e Claude Bernard Lyon 1,
Institut Camille Jordan, CNRS UMR5208,
43 Boulevard du 11 Novembre 1918,
69622 Villeurbanne Cedex, France}
\email{luo@math.univ-lyon1.fr; jintaoluo@foxmail.com}
\thanks{Jintao Luo's contribution to the research presented in this article
was carried out independently before his enrollment at Universit\'e Claude Bernard Lyon 1 and does not form part of his doctoral
thesis project.}

\author{Hang Zhang}
\address{School of Mathematics, Southwest Jiaotong University, Chengdu 610756, China}
\email{zhanghangzh@sina.com; hzhangzh@gmail.com}


\thanks{}

\date{}

\begin{document}
	
	\begin{abstract}
		We study idealized maximal eventually different families through Van Douwen maximality, productivity, and ultrafilter maximality. We show that Van Douwen $\mathcal{J}$-MED families correspond to $(\mathcal{J}\times\emptyset)$-MAD families, and for coanalytic ideals containing the finite sets they are further reduced to infinite analytic $\mathcal{J}$-MAD families. This yields nonexistence results for analytic Van Douwen families for several standard ideals, together with a characterization of the finite case.
		
		We recall the definition of finite productivity and introduce finite-section and $\omega$-centered productivity. Finite-section productivity is equivalent to $\mathcal{U}^*$-maximality for some ultrafilter $\mathcal{U}$ extending $\mathcal{J}^*$, and these notions admit natural Stone-space characterizations. We construct Borel finitely productive families for uniformly weakly Ramsey ideals and Borel $\omega$-centered productive families for uniformly weakly $P^+$ ideals, while no analytic finite-section productive $\mathrm{Fin}$-MED family exists. Finally, for every Ramsey ultrafilter $\mathcal{U}$ and $1\leq\alpha<\omega_1$, there is no analytic $(\mathcal{U}^{\alpha})^*$-MED family, while under $V=L$, there is a $ \boldsymbol{\Sigma}^1_2 $ Ramsey ultrafilter $ \mathcal{U} $ where there is a coanalytic $ \mathcal{U}^* $-MED family.
	\end{abstract}

	\maketitle
	
	\section{Introduction}
	
	It is a classical theorem of Mathias that there is no infinite analytic maximal almost disjoint family on $ \omega $ \cite{mathias1977happy}. While for maximal eventually different families, Horowitz and Shelah constructed a Borel instance in $\bairesp$ \cite{horowitz2024eventually}. A useful way to study this difference is to replace $\mathrm{Fin}$ by a general ideal $\idl{J}$. This point of view has been studied in several recent works; see for example \cite{he2024maximal,he2026ramsey}. The general problem is to understand for which ideals the MED
	behavior remains close to the classical Borel examples, and for which ideals it
	becomes closer to the MAD obstruction.
	
	Let $\idl{J}$ be an ideal on a countable set $D$ and $ \Eq{D}(f,g) = \{n\in D: f(n) = g(n)\} $. The extreme case is when $\idl{J}=\U^*$ for an ultrafilter $\U$. Under this assumption $ f\ E_{\U}\ g \lrtar \Eq{D}(f,g)\in\U $ is an equivalence relation, and a $\U^*$-MED family is exactly a transversal for $E_{\U}$. Equivalently, it chooses one representative from each element of the ultrapower $\omega^D/\U$.
	
	One combinatorial form of ultrafilter maximality is maximality on every positive set. This leads to Van Douwen maximality. A $\idl{J}$-eventually different family is Van Douwen $\idl{J}$-MED if its restriction to every $X\in\idl{J}^+$ is $(\idl{J}\uphar X)$-MED. For $\mathrm{Fin}$, Raghavan proved that Van Douwen MED families exist, but no analytic ones exist \cite{raghavan2010van}. Schrittesser later asked when a closed MED family with full localization coideal can exist in more general product spaces \cite{schrittesser2018compactness}. Related forms of Van Douwen maximality were studied in \cite{schembecker2025van,schembecker2026productivity}. If $\idl{J}=\U^*$, ordinary $\U^*$-MEDness and Van Douwen $\U^*$-MEDness coincide.
	
	Our first results concern definable Van Douwen $\idl{J}$-MED families. Let $ F $ be a partial function from a countable set $ D $ to $ (\omega+1)\setminus\{0\} $, and let 
	\[
	\Ncal_F :=\prod_{n\in\dom(F)}F(n),\ \Sigma_F=\{(n,m):n\in\dom(F),\ m<F(n)\}.
	\]
	A family $\idl{E}\subseteq\prodspc{F}$ is Van Douwen $\idl{J}$-MED exactly when the set of graphs $\gr''\idl{E}$ is a $(\idl{J}\times\emptyset)$-MAD family on $\Sigma_F$. For coanalytic ideals containing all finite sets, we show that an infinite analytic $(\idl{J}\times\emptyset)$-MAD family gives an infinite analytic $\idl{J}$-MAD family on the domain set. In this way, the definability problem for Van Douwen families is reduced to a MAD problem. This gives, in particular, no infinite analytic Van Douwen $\idl{J}$-MED families for the eventually different ideals and the countable iterated Fr\'echet ideals considered in \cite{he2026ramsey}. We also determine the finite case. A finite Van Douwen $\idl{J}$-MED family exists exactly when $F$ is $\idl{J}^*$-constant. So for $\idl{J}=\mathrm{Fin}$, a closed Van Douwen $\mathrm{Fin}$-MED family in $\prodspc{F}$ exists exactly when $F$ is eventually equal to a positive integer, and then the family is finite. This gives the corresponding answer to Schrittesser's question. 
	
	The way of reduction to MAD families is also used for Ramsey ultrafilters. We prove that if $\U$ is Ramsey and $1\leq\alpha<\omega_1$, then there is no infinite analytic $ (\U^\alpha)^*\times\emptyset$-MAD family, and then no analytic $(\U^\alpha)^*$-MED family. Equivalently, $E_{\U^\alpha}$ has no analytic transversal. While, assuming $V=L$, we also construct a $\bfsgm^1_2$ Ramsey ultrafilter carrying a coanalytic $\U^*$-MED family.
	
	Another combinatorial form of ultrafilter maximality is productivity. Finite productivity, in the sense of \cite{schembecker2026productivity}, asks that finitely many functions can be caught at the same time. This definition can be generalized to the idealized case. We also consider its natural countable extension, $ \omega $-centered productive $\idl{J}$-MEDness. That is, for every sequence $\langle h_i:i\in\omega\rangle$, there are $f_i\in\idl{E}$ and one $B\in\idl{J}^+$ such that $ \forall i\in\omega, \ B\subseteq_{\idl{J}}\Eq{D}(h_i,f_i) $. The main notion connecting productivity with ultrafilters is finite-section productivity. A family is finite-section productive $\idl{J}$-MED if each function $h$ is assigned a finite nonempty set $M(h)\subseteq\idl{E}$, and every finite set of functions is caught by functions in the corresponding assignments. Moreover, a $ \idl{J} $-eventually different family $ \idl{E} $ is finite-section productive $ \idl{J} $-MED if and only if there is an ultrafilter $ \U $ such that $ \idl{E} $ is $ \U^* $-MED. Thus finite-section productivity is an internal form of ultrafilter maximality.
	
	These notions have a common Stone-space description. Let
	\[
	\beta_{\idl{J}}D
	=
	\{\U\in\beta D:\idl{J}^*\subseteq\U\},
	\]
	and for $h\in\prodspc{F}$ define
	\[
	O_h(\idl{E})
	=
	\{\U\in\beta_{\idl{J}}D:
	\exists f\in\idl{E}\,
	\Eq{D}(h,f)\in\U\}.
	\]
	Then ordinary $\idl{J}$-MEDness means that every $O_h(\idl{E})$ is nonempty; finite productivity means that the family $\{O_h(\idl{E}):h\in\prodspc{F}\}$ has the finite intersection property; finite-section productivity means that its total intersection is nonempty; and Van Douwen maximality means that every $O_h(\idl{E})$ is dense. For weakly $P^+$ ideals, $ \omega $-centered productivity is described by nonempty intersections of countable subfamilies. This also connects the passage from Van Douwen maximality to finite-section productivity with the relative pseudointersection number $\mdl{p}_{\idl{J}}$ studied in \cite{filipow2022characterizing}.
	
	We also study definable productive families. Building on the Borel idealized MED methods of \cite{he2024maximal} and the productive constructions of \cite{schembecker2026productivity}, we obtain Borel finitely productive families for uniformly weakly Ramsey ideals and Borel $ \omega $-centered productive families for uniformly weakly $P^+$ ideals. In contrast, there is no analytic finite-section productive $\mathrm{Fin}$-MED family. Thus, finite and countable productivity can still occur at the Borel level, while finite-section productivity already has an analytic obstruction.
	
	\section{Preliminaries}
	
	\subsection{Notations}
	
	All countable sets are identified with recursive copies whenever descriptive complexity is involved. If \(A\) is a set and \(f,g\in\omega^{A}\) are functions, define \(\Eq{A}(f,g):=\{n\in A:f(n)=g(n)\}\) and \(\gr(f):=\{(n,f(n)):n\in A\}\). For a subset \(X\subseteq A\times\omega\), define
	\(X(n):=\{m\in\omega:(n,m)\in X\}\) for any $ n\in A $, and define \(\dom(X):=\{n\in A:X(n)\neq\emptyset\}\). We write
	\([X]^{\omega}\) and \([X]^{<\omega}\) for the infinite and finite subsets of
	\(X\). For any $ n\in\omega $, write $ [X]^n $ for subsets of $ X $ with size $ n $. Also, the relation \(B\subseteq^* C\) means that \(B\setminus C\) is finite.
	
	Let $A$ be a set. A collection $\mathcal I \subseteq \mathcal P(A)$ is called a \textbf{proper ideal} on $A$ if it's closed under finite unions and taking subsets, with $\emptyset \in \mathcal I$ and $A \notin \mathcal I$. For an ideal $\mathcal I$, we denote by $\mathcal I^+ = \mathcal P(A) \setminus \mathcal I$ the collection of $\mathcal I$-\textbf{positive} sets, and $ \idl{I}^+ $ is called a \textbf{coideal}. For any proper ideal $ \idl{I} $ and an element $ X\in\idl{I}^+ $, we define the \textbf{restriction of $ \idl{I} $ on $ X $} as $ \idl{I}\uphar X := \{B\subseteq X: B\in\idl{I}\} $. For our purposes, we assume all ideals are proper. A \textbf{proper filter} on $A$ is a collection $\mathcal F \subseteq \mathcal P(A)$ such that its dual $\mathcal{F}^\ast = \{ A \setminus X : X \in \mathcal F \}$ is a proper ideal on $A$. An \textbf{ultrafilter} on $ A $ is a proper filter which contains exactly one of $ X $ and $ A\setminus X $ for any $ X\subseteq A $. An ultrafilter is \textbf{non-principal} if it contains no singletons.
	
	Let $ F $ be a partial function from a countable set $ S $ to $ (\omega+1)\setminus\{0\} $. We always assume that $ \dom(F) $ is infinite. Define the product space $ \Ncal_F $ as
	\[
	\Ncal_F :=\prod_{n\in\dom(F)}F(n),
	\]
	which is non-empty.
	\begin{defn}
		Let $ \idl{I} $ be an ideal on $ \dom(F) $. Two functions $ f, g \in \Ncal_F$ are called \textbf{$ \idl{I} $-eventually different}, if $ \Eq{\dom(F)}(f,g)\in \idl{I} $. We say a family $ \idl{E}\subseteq\Ncal_F $ is \textbf{$ \idl{I} $-eventually different} ($ \idl{I} $-e.d. for short), if any distinct elements in $ \idl{E} $ are $ \idl{I} $-eventually different. An $ \idl{I} $-e.d. family $ \idl{E} $ is called \textbf{maximal} ($ \idl{I}$-MED for short), if for any $ h\in\Ncal_F $, there is some $ f\in \idl{E} $ such that $ \Eq{\dom(F)}(h,f) \in\idl{I}^+$.
	\end{defn}
	
	\begin{defn}
		Let $ \idl{I} $ be a proper ideal on a countable set $ S $. The subsets $A,B\subseteq S$ are called \textbf{$\mathcal I$-almost disjoint} if $A\cap B\in\mathcal I$. A family $\mathcal A\subseteq \idl{I}^+$ is an \textbf{$\mathcal I$-almost disjoint ($ \idl{I}$-a.d.) family}, if its members are pairwise $\mathcal I$-almost disjoint. An \textbf{$\mathcal I$-MAD family} is an $\mathcal I$-a.d. family that is maximal under inclusion, i.e. for every $X\in \idl{I}^+$, there exists $A\in\mathcal A$ such that $X\cap A \in \idl{I}^+$.
	\end{defn}
	
	A coideal $ \idl{H} $ on $ \omega$ is called \textbf{selective}, if it contains no finite subsets, and for any sequence $ H_0\supseteq H_1 \supseteq \cdots $ in $ \idl{H} $, there is a sequence of natural numbers $ k_0 < k_1 <\cdots $ such that $ \{k_0, k_1, \cdots \}\in \idl{H} $ and $ k_{i+1}\in H_{k_i} $ for every $ i\in\omega $. A non-principal ultrafilter \(\U\) on \(\omega\) is called \textbf{Ramsey} if it is selective.
	
	The following theorem is well-known.
	
	\begin{thm}[{\cite[Theorem~9.31]{mathias1977happy}}]\label{thm:ramsey_ultrafilter}
		For a non-principal ultrafilter $ U $ on $ \omega $, it's Ramsey if and only if any one of the following is true:
		\begin{enumerate}[label=(\arabic*),font=\upshape]
			\item For every coloring \(c:[\omega]^2 \to 2\), there is a homogeneous set in \(\U\), i.e. an $ X\in \idl{U} $ such that $ c\uphar [X]^2 $ is constant.
			\item Every analytic subset of $ \ramseysp $ is completely $ \idl{U} $-Ramsey.
		\end{enumerate}
	\end{thm}
	Note that for a coideal $ \idl{H} $, a subset $ \idl{X}\subseteq \ramseysp $ is called \textbf{completely $ \idl{H} $-Ramsey} if for any cube $ [s,S] :=\{X\in \ramseysp: s\subseteq X\subseteq s\cup S\}$ where $ \max s <\min S $ and $ S\in \idl{H} $, there is a subset $ T\subseteq S$ such that $ T\in \idl{H} $ and either $ [s,T]\subseteq \idl{X} $ or $ [s,T]\cap \idl{X} =\emptyset$.
	
	\subsection{Iterated Products}
	
	For each countable ordinal $\alpha < \omega_1$ with $\alpha \ge 1$, fix a non-decreasing sequence $\langle \gamma_n^\alpha : n \in \omega \rangle$ of ordinals less than $\alpha$ such that $\bigcup_{n \in \omega} (\gamma_n^\alpha + 1) = \alpha$ as follows. For successor ordinals $\alpha = \beta+1$ we take $\gamma_n^\alpha = \beta$ for all $n$, and for limit ordinals $\alpha$ we take $\gamma_n^\alpha \geq 1$ to be strictly increasing and cofinal in $\alpha$.
	
	For $ \alpha <\omega_1 $, the sets $S_\alpha \subseteq \omega^{<\omega}$ are defined recursively as follows. Define $ S_0 = \{\emptyset\} $ and for $\alpha \geq 1$, define
	\[
	S_\alpha = \{ n^{\frown} s : n \in \omega,\ s \in S_{\gamma_n^\alpha} \}.
	\]
	
	Note that each $S_\alpha$ is countable and hence can be identified with $\omega$ via a fixed bijection. In particular, $ S_1=\omega $.
	
	For any $X \subseteq S_\alpha$, and for any finite sequence $s \in \omega^{<\omega}$, let
	\[
	X(s) = \{ t \in \omega^{<\omega} : s^{\frown} t \in X \}
	\]
	denote the \textbf{section} of $X$ at $s$. In particular, $X(\emptyset) = X$, and $ X(s) = \{\emptyset\} $ for $ s\in X $. We denote $X(\langle n \rangle) $ by $ X(n)$ for any $ n\in\omega $.
	
	Let 
	\[
	\mathrm{dom}(X) = \{ n \in \omega : X( n ) \neq \emptyset \}
	\]
	be the \textbf{domain} of $ X $, i.e. the set of indices with nonempty section. Note that for $ X\subseteq S_1 = \omega $, we can identify $ \dom(X) = X $, since in this situation $ n\in X $ if and only if $ X(n)\not=\emptyset $.
	
	\begin{defn}
		Let $ \idl{I} $ and $ \idl{J} $ be ideals on $ \omega $. Their \textbf{Fubini product} $ \idl{I}\times\idl{J} $ is the ideal on $ \omega\times\omega $ defined by
		\[
		\idl{I}\times\idl{J}:=\{A\subseteq\omega\times\omega:\{j\in\omega:A(j)\notin\idl{J}\}\in\idl{I}\}.
		\]
		
		For any ideal $ \idl{I} $ and filter $ \idl{F} $, and for each countable ordinal $\alpha < \omega_1$, we define their \textbf{$ \alpha $-th iterated Fubini product} $\idl{I}^\alpha$ and $ \idl{F}^\alpha $ on $S_\alpha$ recursively as follows.
		
		\begin{itemize}
			\item[(1)] For $\alpha = 0$, let
			\[
			\idl{I}^0 :=\{\emptyset\}, \ \idl{F}^0 := \{S_0\}.
			\]
			
			\item[(2)] For $\alpha \ge 1$, define
			\[
			\begin{aligned}
				\idl{I}^\alpha:=&\{X \subseteq S_\alpha: \{ n \in \omega : X(n) \notin \idl{I}^{\gamma_n^\alpha} \} \in \idl{I} \}\\
				\idl{F}^\alpha:=&\{X \subseteq S_\alpha: \{ n \in \omega : X(n) \in \idl{F}^{\gamma_n^\alpha} \} \in \idl{F} \}.
			\end{aligned}
			\]
		\end{itemize}
		
		For example, $ \idl{I}^1 = \idl{I} $ and $ \idl{F}^1  = \idl{F}$.
	\end{defn}	
	
	For our purpose, given a partial function $ F $ as above, we define the graph of $ \idl{N}_F $ by
	\[
	\Sigma_F :=\{(n,m): n\in \dom(F), m < F(n)\},
	\]
	and for any ideal $ \idl{I} $ on $ \dom(F)$, define the ideal $ \idl{I}\times \emptyset $ on $ \Sigma_F $ by
	\[
	\idl{I}\times \emptyset :=\{X\subseteq \Sigma_F: \dom(X)\in \idl{I}\}.
	\]

		\begin{fact}\label{fact:fubini_power_ultrafilter}
		Let $ \U $ be an ultrafilter on $ \omega $. Then $ \U^\alpha $ is an ultrafilter on $ S_\alpha $ for every $ \alpha<\omega_1 $. Consequently, $ ((\U^\alpha)^*)^+=\U^\alpha $.
	\end{fact}
	\begin{proof}
		We argue by induction on $ \alpha $. The case $ \alpha=0 $ is immediate. Suppose $ \alpha\geq1 $ and the conclusion holds below $ \alpha $. For $ X\subseteq S_\alpha $, let $ A_X=\{n\in\omega:X(n)\in\U^{\gamma_n^\alpha}\} $. By the induction hypothesis, for every $ n\in\omega $, either $ X(n) \in \U^{\gamma_n^\alpha} $ or $ \lb S_\alpha \setminus X\rb (n) =  S_{\gamma_n^\alpha}\setminus X(n) \in \U^{\gamma_n^\alpha}$. So $ A_{S_\alpha\setminus X}=\omega\setminus A_X $. Since $ \U $ is an ultrafilter, exactly one of $ A_X $ and $ \omega\setminus A_X $ belongs to $ \U $. By the definition of the iterated Fubini product, exactly one of $ X $ and $ S_\alpha\setminus X $ belongs to $ \U^\alpha $. Thus $ \U^\alpha $ is an ultrafilter.
	\end{proof}

	\section{Definable Van Douwen $ \idl{J} $-MED Families}
	
	We use the terminology in \cite{schembecker2026productivity} for the definition of Van Douwen families. 
	
	\begin{defn}
		Let $ \idl{N}_F $ be a product space, and $ \idl{J} $ be an ideal on $ \dom(F) $. For a $ \idl{J} $-e.d. family $ \idl{E}\subseteq \idl{N}_F $, define
		\[
		\idl{I}_0^\idl{J}(\idl{E}) := \prst(\dom(F) ) \setminus \{X\in \idl{J}^+ : \{g\uphar X: g\in \idl{E}\} \text{ is } (\idl{J}\uphar X) \text{-MED in } \idl{N}_{F\uphar X}\}.
		\]
		The family $ \idl{E} $ is called \textbf{Van Douwen $ \idl{J} $-maximal eventually different} (VD $ \idl{J} $-MED for short), if $ \idl{I}_0^\idl{J} (\idl{E}) = \idl{J} $.
	\end{defn}
	
	We first characterize exactly when a finite VD family exists.
	
	\begin{defn}
		Let $ F $ be a partial function, let $ D=\dom(F) $, and let $ \idl{J} $ be an ideal on $ D $. For $ 1\leq k<\omega $, we say that $ F $ is \textbf{$ \idl{J}^* $-constant with value $ k $}, if $\{n\in D:F(n)= k\}\in\idl{J}^*$. We say that $ F $ is \textbf{$ \idl{J}^* $-constant}, if it is $ \idl{J}^* $-constant with value $ k $ for some $ 1\leq k<\omega $.
	\end{defn}
	
	\begin{prop}\label{prop:jstar_const_fin_vd}
		Let $ F $ be a partial function from a countable set $ S $ to $ (\omega+1)\setminus\{0\} $ such that $D = \dom(F) $ is infinite, as assumed. Let $ \idl{J} $ be an ideal on $ D $. For every $ 1\leq k<\omega $, the following are equivalent.
		\begin{enumerate}[label=(\arabic*),font=\upshape]
			\item There is a VD $ \idl{J} $-MED family in $ \prodspc{F} $ of size $ k $.
			\item The function $ F $ is $ \idl{J}^* $-constant with value $ k $.
		\end{enumerate}
	\end{prop}
	\begin{proof}
		\phantom{a}
		
		\begin{itemize}
			\item[(1)] If $ \idl{E}=\{f_i:i<k\} $ is a VD $ \idl{J} $-MED family. Since $ \idl{E} $ is $ \idl{J} $-e.d., we have
			\[
			B_0:=\bigcup_{i<j<k}\Eq{D}(f_i,f_j)\in\idl{J}.
			\]
			Let
			\[
			B_1:=\{n\in D:F(n)\setminus\{f_i(n):i<k\}\neq\emptyset\}.
			\]
			We claim that $ B_1\in\idl{J} $. Otherwise, define $ h\in\prodspc{F\uphar B_1} $ by $h(n):=\min(F(n)\setminus\{f_i(n):i<k\})$ for every $ n\in B_1 $. Then for any $ i<k $ we have $\Eq{B_1}(h,f_i)=\emptyset$, contradicting the VD $ \idl{J} $-maximality of $ \idl{E} $.
			
			Now we have that $ D\setminus(B_0 \cup B_1) \in \idl{J}^* $. For every $ n\in D\setminus(B_0 \cup B_1) $, we have that $\forall i<j,  \ f_i (n)\not= f_j (n)$, and $ F(n) \setminus\{f_i(n):i<k\} =\emptyset$. So $ F(n)=k $, and then
			\[
			\idl{J}^* \ni D\setminus(B_0\cup B_1)\subseteq \{n\in D:F(n)= k\},
			\]
			i.e. $ F  $ is $ \idl{J}^* $-constant with value $ k $.
			\item[(2)] Conversely, suppose that $ C:=\{n\in D:F(n)=k\} $ has $ D\setminus C\in\idl{J} $. For every $ i<k $, define $ f_i\in\prodspc{F} $ by
			\[
			f_i(n)=
			\begin{cases}
				i,&n\in C,\\
				0,&n\in D\setminus C.
			\end{cases}
			\]
			Note that $ F(n)>0 $ for every $ n\in D $. 
			
			\begin{itemize}
				\item[i.] For distinct $ i,j<k $, we have $ \Eq{D}(f_i,f_j)\subseteq D\setminus C\in\idl{J} $, so $ \{f_i:i<k\} $ is $ \idl{J} $-e.d..
				\item[ii.] To show VD maximality, let $ X\in\idl{J}^+ $ and $ h\in\prodspc{F\uphar X} $. We have $ X\cap C\in\idl{J}^+ $. Note that
				\[
				X\cap C=\bigsqcup_{i<k}\{n\in X\cap C:h(n)=i\}\in \idl{J}^+.
				\]
				So there is some $ i<k $ such that $ \{n\in X\cap C:h(n)=i\}\in \idl{J}^+ $, which means $ \Eq{X}(h,f_i)\in \idl{J}^+ $. Thus the family $ \{f_i:i<k\} $ is VD $ \idl{J} $-MED.
			\end{itemize}
		\end{itemize}
	\end{proof}
	
	\begin{defn}
		Let $ F $ be a partial function, let $ D=\dom(F) $, and let $ \idl{J} $ be an ideal on $ D $. We say that $ F $ \textbf{$ \idl{J}^* $-converges to $ \omega $}, if $ \{n\in D:F(n)>k\}\in\idl{J}^* $ for every $ k\in\omega $.
	\end{defn}
	
	\begin{prop}\label{prop:infinite_jed_convergence}
		Let $ F $ be a partial function, let $ D=\dom(F) $, and let $ \idl{J} $ be an ideal on $ D $. If there is an infinite $ \idl{J} $-e.d. family in $ \prodspc{F} $, then $ F $ $ \idl{J}^* $-converges to $ \omega $.
	\end{prop}
	\begin{proof}
		Let $ \idl{E}\subseteq\prodspc{F} $ be an infinite $ \idl{J} $-e.d. family. For any $ k>0 $ define $ X_k:=\{n\in D:F(n)\leq k\}$. Choose distinct functions $ f_i\in\idl{E} $ for $ i\leq k $. For every $ n\in X_k $, there are $ k+1 $ values $ f_0(n),\ldots,f_k(n) $ in the set $ F(n) $ of size at most $ k $. So there will be some $ i,j\leq k $ such that $ f_i(n) = f_j(n) $, i.e.
		\[
		X_k\subseteq\bigcup_{i<j\leq k}\Eq{D}(f_i,f_j)\in\idl{J},
		\]
		So $ X_k\in\idl{J} $ for every $ 1\leq k<\omega $. The case $ k=0 $ is trivial since $ F(n)>0 $ for every $ n\in D $.
	\end{proof}
	
	\begin{prop}\label{prop:vdj_jtimesempty}
		Let $ F $ be a partial function and $ \idl{J} $ be an ideal on $ \dom(F) $. Let $ \idl{E} $ be a subset of $ \prodspc{F} $. The following are equivalent.
		\begin{enumerate}[label=(\arabic*),font=\upshape]
			\item $\idl{E}$ is a VD $ \idl{J} $-MED family in $ \prodspc{F} $.
			\item The continuous image $\gr''\idl{E}$ of $ \idl{E} $ under the function $\gr: \prodspc{F} \to \prst(\Sigma_F) $ is a $ \idl{J}\times \emptyset $-MAD family on $ \Sigma_F $.
		\end{enumerate}
	\end{prop}	
	\begin{proof}
		\phantom{a}
		
		\begin{itemize}
			\item[(1)] Let $\idl{E}$ be VD $ \idl{J} $-MED. 
			\begin{itemize}
				\item[i.] For every $ f\in \idl{E} $, we have $ \dom(\gr(f)) = \dom(F)\notin \idl{J} $ by properness, so $ \gr''\idl{E}\subseteq (\idl{J}\times \emptyset)^+ $.
				\item[ii.] We show that $ \gr''\idl{E} $ is $\idl{J}\times\emptyset$-almost disjoint. For distinct $ f, g\in \idl{E} $, we have that
				\[
				\dom(\gr(f)\cap \gr(g)) = \Eq{\dom(F)}(f,g),
				\]
				which is in $ \idl{J} $ as $ \idl{E} $ is $ \idl{J} $-e.d.. So $ \gr''\idl{E} $ is $ \idl{J}\times \emptyset $-a.d..
				\item[iii.] To show maximality, for any $ X\in (\idl{J}\times\emptyset)^+ $, define $ h_X \in \prodspc{F\uphar \dom(X)} $ by $ h_X(n) = \min X(n) $ for any $ n\in\dom(X) $. Since \(\idl{E}\) is VD $\idl{J}$-MED, there exists $f\in\idl{E}$ such that $\Eq{\dom(X)}(h_X,f)\in\idl{J}^{+}$. As
				\[
				\Eq{\dom(X)}(h_X,f) \subseteq \dom(X\cap\gr(f)),
				\]
				we have that $ \dom(X\cap\gr(f))\in\idl{J}^+ $. So $ \gr''\idl{E} $ is $ \idl{J}\times\emptyset $-maximal.
			\end{itemize}
			This proves that $ \gr''\idl{E} $ is a $\idl{J}\times\emptyset$-MAD family on $\Sigma_F$.
			\item[(2)] Conversely,  assume that $ \gr''\idl{E} $ is $ \idl{J}\times\emptyset $-MAD.
			\begin{itemize}
				\item[i.] To prove $ \idl{J} $-eventual difference, for distinct $ f,g\in\idl{E} $, since $ \gr''\idl{E} $ is $ \idl{J}\times\emptyset $ -a.d., we have $ \gr(f)\cap\gr(g)\in\idl{J}\times\emptyset $. So $ \Eq{\dom(F)}(f,g)=\dom(\gr(f)\cap\gr(g))\in\idl{J} $.
				\item[ii.] To show it's VD $ \idl{J} $-maximal, let $ X\in\idl{J}^+ $ and $ h\in \prodspc{F\uphar X} $. We have that $ \gr(h) = \{(n, h(n)): n\in X\}\in (\idl{J}\times \emptyset )^+ $. By the maximality of $ \gr''\idl{E} $, there exists $ f\in \idl{E} $ such that $ \gr(h) \cap \gr(f)\in (\idl{J}\times \emptyset )^+$. So
				\[
				\Eq{X}(h,f) = \dom(\gr(h)\cap \gr(f))\in\idl{J}^+.
				\]
			\end{itemize}
			This shows that $ \idl{E} $ is VD $ \idl{J} $-MED.
		\end{itemize}
	\end{proof}
	
	We now recall the notion of an almost disjoint refinement with respect to an ideal. We use the terminology in \cite{farkas2018almost}.
	
	\begin{defn}
		Let $ \idl{J} $ be an ideal on a countable set $ S $, and let $ \idl{H}\subseteq\idl{J}^+ $. A family
		\[
		\idl{A}=\{A_H:H\in\idl{H}\}
		\]
		is called a \textbf{$ \idl{J} $-almost disjoint refinement} ($ \idl{J} $-ADR for short) of $ \idl{H} $, if
		\begin{enumerate}[label=(\arabic*),font=\upshape]
			\item $ A_H\subseteq H $ and $ A_H\in\idl{J}^+ $ for every $ H\in\idl{H} $;
			\item $ A_H\cap A_K\in\idl{J} $ for all distinct $ H,K\in\idl{H} $.
		\end{enumerate}
	\end{defn}
	
	\begin{defn}
		Let $ \idl{J} $ be an ideal on a countable set $ S $.
		\begin{enumerate}[label=(\arabic*),font=\upshape]
			\item We say that $ \idl{J} $ has the \textbf{countable $ \idl{J} $-almost disjoint refinement property} $ \idl{J}$-$\operatorname{ADR}_{\omega} $, if every countable subfamily of $ \idl{J}^+ $ has a $ \idl{J} $-ADR.
			\item We say that $ \idl{J} $ has the \textbf{countable disjoint refinement property} $ \operatorname{DR}_{\omega} $, if for every sequence $ \langle X_n:n\in\omega\rangle $ in $ \idl{J}^+ $, there is a sequence $ \langle Y_n:n\in\omega\rangle $ in $ \idl{J}^+ $ such that $ Y_n\subseteq X_n $ for every $ n\in\omega $, and the sets $ Y_n $ are pairwise disjoint.
		\end{enumerate}
	\end{defn}
	
	Note that for $ \idl{J} $ containing all finite subsets of the domain, the $ \idl{J}$-$\operatorname{ADR}_{\omega} $ property is equivalent to its sequence version, which is furthermore equivalent to the quotient Boolean algebra $ \prst(S)/\idl{J} $ having $ \operatorname{Rp}(\omega) $ in the sense of \cite{balcar1989disjoint}.
	
	\begin{fact}\label{fact:adr_w_iff_de_w}
		Let $ \idl{J} $ be an ideal on a countable set $ S $ containing all finite subsets of $ S $. Then $ \idl{J} $ has $ \idl{J} $-$ \operatorname{ADR}_{\omega} $ if and only if $ \idl{J} $ has $ \operatorname{DR}_{\omega} $.
	\end{fact}
	\begin{proof}
		\phantom{a}
		
		\begin{itemize}
			\item[(1)] We first show that (1) implies (2). Let $ \langle X_n:n\in\omega\rangle $ be a sequence in $ \idl{J}^+ $. To prevent repetition, we recursively choose finite sets $ s_n\subseteq X_n $ such that the sets $ H_n:=X_n\setminus s_n \in \idl{J}^+ $ are pairwise distinct. Let $ \{Z_n:n\in\omega\} $ be a $ \idl{J} $-ADR of the countable family $ \{H_n:n\in\omega\} $, where $ Z_n\subseteq H_n $. Define $ Y_n:=Z_n\setminus\bigcup_{i<n}Z_i $, which are pairwise disjoint, and we have that $ Y_n \subseteq X_n, \ Y_n \in \idl{J}^+ $ for all $ n\in \omega $. So $ \langle Y_n:n\in\omega\rangle $ witnesses $ \operatorname{DR}_{\omega} $.
			\item[(2)] Conversely, supposing that (2) holds, we show (1). Let $ \idl{H}\subseteq\idl{J}^+ $ be countable. If $ \idl{H}=\emptyset $, then it is trivial. If not, choose a sequence $ \langle H_n:n\in\omega\rangle $ in $ \idl{H} $ such that every element of $ \idl{H} $ occurs in the sequence. Let $ \langle Y_n:n\in\omega\rangle $ witness $ \operatorname{DR}_{\omega} $ for this sequence. For every $ H\in\idl{H} $, let $ n_H $ be the least $ n $ such that $ H_n=H $, and define $ A_H:=Y_{n_H} $. Then $ \{A_H:H\in\idl{H}\} $ is a $ \idl{J} $-ADR of $ \idl{H} $.
		\end{itemize}
	\end{proof}
	
	The following lemma is immediate from \cite[Proposition~1.9]{farkas2018almost}, but we state it for our convenience.
	
	\begin{lem}\label{lem:coanalytic_dr_omega}
		Let $ \idl{J} $ be a coanalytic ideal on a countable set $ S $ such that $ \finsub{S}\subseteq\idl{J} $. Then $ \idl{J} $ has $ \idl{J}$-$\operatorname{ADR}_{\omega} $, and hence it has $ \operatorname{DR}_{\omega} $.
	\end{lem}
	\begin{proof}
		For every $ X\in\idl{J}^+ $, the ideal $ \idl{J}\uphar X $ is a coanalytic subset of $ \prst(X) $, so it has Baire property and then $ \idl{J}\uphar X $ is meager in $ \prst(X) $, by a well known result of Sierpi\'nski. This means $ \idl{J} $ is everywhere meager, and by \cite[Proposition~1.9]{farkas2018almost}, every countable subfamily of $ \idl{J}^+ $ has a $ \idl{J} $-ADR. Thus $ \idl{J} $ has $ \idl{J} $-$ \operatorname{ADR}_{\omega} $, and it has $ \operatorname{DR}_{\omega} $ by Fact \ref{fact:adr_w_iff_de_w}.
	\end{proof}
	
	We can now transfer analytic MAD families from $ \idl{J}\times\emptyset $ to $ \idl{J} $. The proof below is motivated by the trace argument for Van Douwen
	families in \cite{raghavan2010van}; see also
	\cite[Proposition~2.6 and Corollary~2.7]
	{schembecker2025van}.
	
	\begin{thm}\label{thm:jtimesemptyset_to_j}
		Let $ F $ be a partial function, let $ D=\dom(F) $, and let $ \idl{J} $ be a coanalytic ideal on $ D $ such that $ \finsub{D}\subseteq\idl{J} $. If there is an infinite analytic $ \idl{J}\times\emptyset $-MAD family on $ \Sigma_F $, then there is an infinite analytic $ \idl{J} $-MAD family on $ D $.
	\end{thm}
	\begin{proof}
		Let $ \idl{A}\subseteq\prst(\Sigma_F) $ be an infinite analytic $ \idl{J}\times\emptyset $-MAD family, and fix a sequence $ \langle A_n:n\in\omega\rangle $ of distinct elements in $ \idl{A} $. For every $ n\in\omega $, we have $ \dom(A_n)\in\idl{J}^+ $. By Lemma \ref{lem:coanalytic_dr_omega}, there are pairwise disjoint sets $ Y_n\in\idl{J}^+ $ such that $ Y_n\subseteq\dom(A_n) $ for every $ n\in\omega $. Define $ g\in\prodspc{F} $ by
		\[
		g(k)=
		\begin{cases}
			\min A_n(k),& k\in Y_n,\\
			0,& k\in D\setminus\bigcup_{n\in\omega}Y_n,
		\end{cases}
		\]
		and let $ G=\gr(g) $. 
		
		For any $ A\subseteq\Sigma_F $, define
		\[
		\operatorname{tr}_G(A):=\dom(A\cap G)=\{k\in D:(k,g(k))\in A\}.
		\]
		The function $ \operatorname{tr}_G:\prst(\Sigma_F)\to\prst(D) $ is continuous. Define
		\[
		\idl{B}:=\operatorname{tr}_G''\{A\in\idl{A}:\operatorname{tr}_G(A)\in\idl{J}^+\}.
		\]
		Note that $ \idl{B} $ is analytic.
		
		We show that $ \idl{B} $ is an infinite $ \idl{J} $-MAD family. Note that $ \idl{B}\subseteq \idl{J}^+$.
		\begin{itemize}
			\item[(1)] We show that $ \idl{B} $ is infinite. For every $ n\in\omega $, we have $ Y_n\subseteq\trace_G(A_n) $, so $ \trace_G(A_n)\in\idl{J}^+ $ and then $ \trace_G(A_n)\in\idl{B} $. Also, for distinct $ n,m $ we have that $ \trace_G(A_n)\cap \trace_G(A_m) = \dom(A_n \cap A_m \cap G)\in \idl{J}$, so they are also distinct. 
			\item[(2)] Similarly, for any distinct $ A,A'\in\idl{A} $, we have that $ \trace_G(A)\cap \trace_G(A') = \dom(A \cap A' \cap G)\in \idl{J}$, so $ \idl{B} $ is $ \idl{J} $-a.d.
			\item[(3)] We now show maximality. Let $ X\in \idl{J}^+ $. Then $ \gr(g\uphar X)\in(\idl{J}\times\emptyset)^+ $. By the maximality of $ \idl{A} $, there exists $ A\in\idl{A} $ such that $ \gr(g\uphar X)\cap A $ is positive. As $ \dom(\gr(g\uphar X)\cap A)=X\cap\operatorname{tr}_G(A) $, we have $ X\cap\operatorname{tr}_G(A)\in\idl{J}^+ $. So $ \operatorname{tr}_G(A)\in\idl{B} $ and catches $ X $. This way $ \idl{B} $ is a $ \idl{J}$-MAD family on $ D $.
		\end{itemize}
	\end{proof}
	
	We note that the converse also holds. Suppose $ F $ is a partial function and $ \idl{J} $ is an ideal. Let $ \idl{A} $ be $ \idl{J} $-MAD. Consider
	\[
	\{\{(n,m):n\in A \ \wedge \ m<F(n)\}\subseteq \Sigma_F: A\in \idl{A}\} . 
	\]
	This is a $ \idl{J}\times\emptyset $-MAD family.
	
	As an application, we obtain the following results.
	
	\begin{cor}
		Let $ \idl{J} $ be one of the following ideals:
		\[
		\evendiff_{\mathrm{fin}},\qquad \evendiff,\qquad \finalpha{\alpha}\text{ for }1\leq\alpha<\omega_1.
		\]
		Let $ F $ be a partial function whose domain is the underlying set of $ \idl{J} $. Then there is no infinite analytic $ \idl{J}\times\emptyset $-MAD family on $ \Sigma_F $.
	\end{cor}
	\begin{proof}
		If there were an infinite analytic $ \idl{J}\times\emptyset $-MAD family on $ \Sigma_F $, then by Theorem \ref{thm:jtimesemptyset_to_j} there would be an infinite analytic $ \idl{J} $-MAD family. This contradicts \cite[Corollaries~6.9 and 6.19]{he2026ramsey}.
	\end{proof}
	
	\begin{cor}\label{cor:no_analytic_vd_standard_ideals}
		Let $ \idl{J} $ be one of $ \evendiff_{\mathrm{fin}} $, $ \evendiff $, and $ \finalpha{\alpha} $ for $ 1\leq\alpha<\omega_1 $, and let $ F $ be a partial function whose domain is the underlying set of $ \idl{J} $. Then there is no infinite analytic VD $ \idl{J} $-MED family in $ \prodspc{F} $. In particular, there is no infinite closed VD $ \idl{J} $-MED family in $ \prodspc{F} $.
	\end{cor}
	\begin{proof}
		Suppose that $ \idl{E}\subseteq\prodspc{F} $ is an infinite analytic VD $ \idl{J} $-MED family. By Proposition \ref{prop:vdj_jtimesempty}, $ \gr''\idl{E} $ is a $ \idl{J}\times\emptyset $-MAD family on $ \Sigma_F $. Since $ \gr $ is injective and continuous, $ \gr''\idl{E} $ is infinite and analytic, a contradiction.
	\end{proof}
	
	We can summarize the above results as the following.
	
	\begin{thm}\label{thm:vd_layers}
		Let $ F $ be a partial function, let $ D=\dom(F) $, and let $ \idl{J} $ be a coanalytic ideal on $ D $ such that $ \finsub{D}\subseteq\idl{J} $. Then the following hold.
		\begin{enumerate}[label=(\arabic*),font=\upshape]
			\item The function $ F $ is $ \idl{J}^* $-constant if and only if there is a finite VD $ \idl{J} $-MED family in $ \prodspc{F} $.
			\item Suppose that $ F $ $ \idl{J}^* $-converges to $ \omega $ (so $ F $ is not $ \idl{J}^* $-constant). If there is no infinite analytic $ \idl{J} $-MAD family on $ D $, then there is no analytic VD $ \idl{J} $-MED family in $ \prodspc{F} $.
			\item If $ F $ is not $ \idl{J}^* $-constant and does not $ \idl{J}^* $-converge to $ \omega $, then there are no VD $ \idl{J} $-MED families in $ \prodspc{F} $.
			
		\end{enumerate}
	\end{thm}
	
	\begin{proof}
		\phantom{a}
		
		\begin{itemize}
			\item[(1)] This is Proposition \ref{prop:jstar_const_fin_vd}.
			\item[(2)] In this condition, $ F $ is not $ \idl{J}^* $-constant, so there is no finite VD $ \idl{J} $-MED family. If there is an infinite analytic VD $ \idl{J} $-MED family, then by Proposition \ref{prop:vdj_jtimesempty} and Theorem \ref{thm:jtimesemptyset_to_j}, there will be an infinite analytic $ \idl{J} $-MAD family, a contradiction.
			\item[(3)] By Proposition \ref{prop:jstar_const_fin_vd} and \ref{prop:infinite_jed_convergence}.
		\end{itemize}
	\end{proof}
	
	\begin{cor}
		Let $ \idl{J} $ be one of $ \evendiff_{\mathrm{fin}} $, $ \evendiff $, and $ \finalpha{\alpha} $ for $ 1\leq\alpha<\omega_1 $, and let $ F $ be a partial function whose domain is the underlying set of $ \idl{J} $.
		\begin{enumerate}[label=(\arabic*),font=\upshape]
			\item If $ F $ is $ \idl{J}^* $-constant, then there is a finite closed VD $ \idl{J} $-MED family in $ \prodspc{F} $.
			\item If $ F $ is not $ \idl{J}^* $-constant, then there is no analytic VD $ \idl{J} $-MED family in $ \prodspc{F} $.
		\end{enumerate}
	\end{cor}
	\begin{proof}
		This follows by \cite[Corollaries~6.9 and 6.19]{he2026ramsey} and Theorem \ref{thm:vd_layers}.
	\end{proof}

	In particular, we take $ \idl{J}=\mathrm{Fin} $. For the function $ F:\omega\to(\omega+1)\setminus\{0\} $, there is a closed VD $ \mathrm{Fin} $-MED family in $ \prodspc{F} $ if and only if $ F $ is eventually equal to some positive natural number.
	
	\section{Definable Productive Families}
	
	Let $ F $ be a partial function from a countable set $ S $ to
	$ (\omega+1)\setminus\{0\} $, let $ D=\dom(F) $, and let $ \idl{J} $ be an
	ideal on $ D $. For $ A,B\subseteq D $, write $ A\subseteq_{\idl{J}}B $ if
	$ A\setminus B\in\idl{J} $. Throughout this section, $ \idl{E} $ denotes a
	$ \idl{J} $-e.d. subfamily of $ \prodspc{F} $.
	
	\subsection{Productivity and weakly $ P^+ $ ideals}
	
	We first introduce three simultaneous strengthenings of productive maximality.
	
	\begin{defn}\label{def:productive_jmed}
		Let $ \idl{E}\subseteq\prodspc{F} $ be $ \idl{J} $-e.d..
		\begin{enumerate}[label=(\arabic*),font=\upshape]
			\item The family $ \idl{E} $ is called \textbf{finitely productive}
			(\FP $ \idl{J} $-MED for short), if for every $ 1\leq k<\omega $ and
			every $ h_0,\ldots,h_{k-1}\in\prodspc{F} $, there are
			$ f_0,\ldots,f_{k-1}\in\idl{E} $ such that
			\[
			\bigcap_{i<k}\Eq{D}(h_i,f_i)\in\idl{J}^+.
			\]
			\item The family $ \idl{E} $ is called \textbf{finite-section productive}
			(\FSP $ \idl{J} $-MED for short), if there is a function
			$ M:\prodspc{F}\to\finsub{\idl{E}}\setminus\{\emptyset\} $ such that,
			for every $ 1\leq k<\omega $ and every
			$ h_0,\ldots,h_{k-1}\in\prodspc{F} $, there are
			$ f_i\in M(h_i) $ for $ i<k $ so that
			$ \bigcap_{i<k}\Eq{D}(h_i,f_i)\in\idl{J}^+ $.
			\item The family $ \idl{E} $ is called \textbf{$ \omega $-centered
				productive} (\omegaCP $ \idl{J} $-MED for short), if for every sequence
			$ \la h_i:i\in\omega\ra $ in $ \prodspc{F} $, there are a sequence
			$ \la f_i:i\in\omega\ra $ in $ \idl{E} $ and a set
			$ B\in\idl{J}^+ $ such that
			$ B\subseteq_{\idl{J}}\Eq{D}(h_i,f_i) $ for every $ i\in\omega $.
		\end{enumerate}
	\end{defn}
	
	We have the following implication from their definitions.
	
	\begin{prop}\label{prop:elementary_productive_implications}
		Every \FSP $ \idl{J} $-MED family is \FP $ \idl{J} $-MED. Every
		\omegaCP $ \idl{J} $-MED family is \FP $ \idl{J} $-MED, and every
		\FP $ \idl{J} $-MED family is $ \idl{J} $-MED.
	\end{prop}
	\begin{proof}
		The implication from \FSP to \FP is by definition. Suppose that
		$ \idl{E} $ is \omegaCP and fix $ h_0,\ldots,h_{k-1}\in\prodspc{F} $. Make this tuple a countable sequence and choose $ B\in\idl{J}^+ $ and
		$ f_i\in\idl{E} $ by the definition of \omegaCP $ \idl{J} $-MEDness. As $ B\setminus\bigcap_{i<k}\Eq{D}(h_i,f_i) = \bigcup_{i<k} B\setminus\Eq{D}(h_i,f_i) \in \idl{J}$, we have that $ \bigcap_{i<k}\Eq{D}(h_i,f_i)\in\idl{J}^+ $. Finally, the
		case $ k=1 $ in finite productivity is ordinary $ \idl{J} $-maximality.
	\end{proof}
	
	\begin{defn}\label{def:weakly_p_plus}
		The ideal $ \idl{J} $ is called \textbf{weakly $ P^+ $}, if every
		decreasing sequence $ X_0\supseteq X_1\supseteq\cdots $ in
		$ \idl{J}^+ $ has a $ \idl{J}^+ $-pseudointersection, that is, a
		set $ B\in\idl{J}^+ $ such that $ B\subseteq_{\idl{J}}X_n $ for every
		$ n\in\omega $.
	\end{defn}

	The terminology of weakly $ P^+ $ property was introduced in \cite{he2024maximal}, and the property has been investigated in the literature. In \cite{filipow2022characterizing}, such an ideal $ \idl{I} $ is called $ P^+ (\idl{I}) $. In \cite{kowitz2026homogeneity}, such an ideal is called $ \hat{P}^+ $.
	
	\begin{prop}\label{prop:weak_p_plus_countable_mad}
		Let $ \idl{J} $ be an ideal on a countably infinite set $ D $. The following are equivalent.
		\begin{enumerate}[label=(\arabic*),font=\upshape]
			\item The ideal $ \idl{J} $ is weakly $ P^+ $.
			\item There is no countably infinite $ \idl{J} $-MAD family on $ D $.
		\end{enumerate}
	\end{prop}
	\begin{proof}
		\phantom{a}
		
		\begin{itemize}
			\item[(1)] If $ \idl{J} $ is weakly $ P^+ $, and let
			$ \idl{A}=\{A_n:n\in\omega\}\subseteq \idl{J}^+ $ be a countably infinite
			$ \idl{J} $-a.d. family. For $ n\in\omega $, let
			$ X_n=D\setminus\bigcup_{i<n}A_i $. So $ \la X_n:n\in\omega\ra $
			is decreasing from $ D $. Also, every $ X_n $ is positive, as $ X_n \supseteq A_n \setminus \bigcup_{i<n}A_i \in\idl{J}^+$. Now, by weak $ P^+ $ property there is a $ B\in\idl{J}^+ $ such that $ B\subseteq_{\idl{J}}X_n $ for all
			$ n $, i.e. $ B\setminus (D\setminus\bigcup_{i<n}A_i) = B \cap \bigcup_{i<n}A_i \in \idl{J} $ for any $ n $. So $ \idl{A} $ is not maximal.
			\item[(2)] Conversely, suppose that $ \idl{J} $ is not weakly $ P^+ $. By definition there is a decreasing sequence $ X_0\supseteq X_1\supseteq\cdots $ in
			$ \idl{J}^+ $ such that there is no $ \idl{J}^+ $-pseudointersection. Let $ D_n=X_n\setminus X_{n+1} $ for $ n\in\omega $, and let
			\[
			\idl{A} = \{D_n: D_n\in \idl{J}^+\}\cup R,
			\]
			where $ R = \{D\setminus X_0\}$, if $ D\setminus X_0 \in \idl{J}^+ $, and $R =  \emptyset $ otherwise. We show that $ \idl{A}\subseteq \idl{J}^+ $ is countably infinite and $ \idl{J} $-MAD.
			\begin{itemize}
				\item[i.] We show $ \idl{A} $ is infinite. Otherwise, there is $ N $ such that $ D_n\in\idl{J} $ for all $ n\geq N $. Note that $ X_N\subseteq X_n $ for $ n\leq N $, and $ X_N\setminus X_n=\bigcup_{k=N}^{n-1}D_k\in\idl{J} $ for all $ n>N $.
				Thus $ X_N $ would be a positive pseudointersection, a contradiction.
				\item[ii.] Note that $ \idl{A} $ is pairwise disjoint.
				\item[iii.] To show maximality, suppose that $ Y\in\idl{J}^+ $ is $ \idl{J} $-a.d. from every member of $ \idl{A} $. Then $ Y\cap (D\setminus X_0)\in\idl{J} $ and $ Y\cap D_k\in\idl{J} $ for every $ k\in\omega $. Note that $ Y\setminus X_n = (Y \cap(D\setminus X_0)) \cup \bigcup_{k<n}(Y\cap D_k) \in \idl{J} $ for any $ n\in\omega $. So $ Y $ is a positive pseudointersection of the sequence, again a contradiction.
			\end{itemize}
			So $ \idl{A} $ is a countably infinite $ \idl{J} $-MAD family.
		\end{itemize}
	\end{proof}
	
	\begin{defn}[{\cite{filipow2022characterizing}}]\label{def:p_j}
		A family $ \idl{B}\subseteq\idl{J}^+ $ is said to have \textbf{$ \idl{J}^+$-finite intersection property} ($ \idl{J}^+ $-FIP for short), if every finite intersection of members of
		$ \idl{B} $ is in $ \idl{J}^+ $. A set $ C\in\idl{J}^+ $ is a
		\textbf{$ \idl{J}^+$-pseudointersection} of $ \idl{B} $, if
		$ C\subseteq_{\idl{J}}B $ for every $ B\in\idl{B} $. Let
		$ \mdl{p}_{\idl{J}} $ be the least cardinality of a
		$ \idl{J}^+ $-FIP family without a $ \idl{J}^+ $-pseudointersection, and let $ \mdl{p}_{\idl{J}}=\mathfrak{c}^+ $ if no such family exists.
	\end{defn}
	
	\begin{prop}\label{prop:borel_p_j_bound}
		Let $ \idl{J} $ be a Borel ideal on a countably infinite set $ D $ such that $ \finsub{D}\subseteq\idl{J} $. Then $\mdl{p}_{\idl{J}}\leq\mathfrak{c}$.
	\end{prop}
	\begin{proof}
		Extend the proper filter $ \idl{J}^* $ to an ultrafilter $ \U $ on $ D $. Since $ \finsub{D}\subseteq\idl{J} $, the ultrafilter $ \U $ is non-principal. Note that $ \U\subseteq\idl{J}^+ $ and then $ \U $ has $ \idl{J}^+ $-FIP.
		
		We claim that $ \U $ has no $ \idl{J}^+ $-pseudointersection. If not, there is $ C\in\idl{J}^+ $ such that $ C\subseteq_{\idl{J}}A $ for all $ A\in\idl{U} $. We show that the proper filter $ (\idl{J}\uphar C)^* \subseteq \prst(C) $ is a non-principal ultrafilter on $ C $.
		\begin{itemize}
			\item[(1)] Let $ X\subseteq C $. Note that either $ X\in\U $ or $ D\setminus X\in\U $.
			\begin{itemize}
				\item[i.] If $ X\in \U $, then by the property of $ C $ we have $ C\setminus X \in \idl{J}\uphar C$. So $ C \setminus (C\setminus X) = X \in (\idl{J}\uphar C)^* $.
				\item[ii.] If $ D\setminus X\in\U $,  again we have $ C\setminus (D\setminus X) = C\cap X = X \in \idl{J}\uphar C $. So $ C\setminus X \in  (\idl{J}\uphar C)^*$. 
			\end{itemize}
			So $ (\idl{J}\uphar C)^* $ is an ultrafilter.
			\item[(2)] To show it's non-principal, if there is a finite $ s\subseteq C $ such that $ s\in (\idl{J}\uphar C)^* $, then we have $ C\setminus s \in \idl{J} $, but $ C\in \idl{J}^+ $, a contradiction.
		\end{itemize}
		
		Now, note that $ (\idl{J}\uphar C)^* $ is a Borel subset of $ \prst(C) $, which is impossible, as there are no Borel non-principal ultrafilters.
		
		Thus $ \U $ is a $ \idl{J}^+ $-FIP family without a
		$ \idl{J}^+ $-pseudointersection. Since
		$ |\U|\leq|\prst(D)|=\mathfrak c $, we conclude that
		$ \mdl{p}_{\idl{J}}\leq\mathfrak c $.
	\end{proof}
	
	The following proposition is mentioned in \cite{filipow2022characterizing}, and we give a proof for self-containment.
	
	\begin{prop}\label{prop:weak_p_plus_pj}
		The ideal $ \idl{J} $ is weakly $ P^+ $ if and only if
		$ \mdl{p}_{\idl{J}}>\omega $.
	\end{prop}
	\begin{proof}
		If $ \idl{J} $ is weakly $ P^+ $, for any countable $ \idl{J}^+ $-FIP subfamily $ \{B_n:n\in\omega\} $, the sequence $ X_n=\bigcap_{i\leq n}B_i $ is decreasing and positive. By weak $ \idl{P}^+ $ property we have a $ C\in\idl{J}^+ $, such that  $ C\subseteq_{\idl{J}} X_n$ for all $ n $, which is also a positive pseudointersection of the original subfamily. Conversely, if $ \mdl{p}_{\idl{J}}>\omega $, as decreasing sequences always have $ \idl{J}^+ $-FIP, we get a positive pseudointersection.
	\end{proof}
	
	\subsection{A framework of productive \texorpdfstring{$ \idl{J} $}{J}-MED families}
	
	We now isolate a general mechanism for constructing Borel productive families. The constructions in this subsection are motivated by the Borel idealized MED construction and uniform weak Ramsey methods of \cite[Sections~3--5]{he2024maximal}, and by the Borel productive constructions of \cite[Remark~2.8 and Section~4]{schembecker2026productivity}. In particular, when $ \idl{J}=\mathrm{Fin} $, our notion of \FP maximality is exactly finite productivity in the sense of \cite[Remark~2.8]{schembecker2026productivity}.
	
	\begin{lem}\label{lem:protected_splicing}
		Let $ Q $ and $ P $ be standard Borel spaces. Suppose that the following
		Borel data are given.
		\begin{enumerate}[label=(\arabic*),font=\upshape]
			\item A Borel map
			$ q\mapsto\la h_i^q:i\in\omega\ra $ from $ Q $ to
			$ (\prodspc{F})^\omega $.
			\item A Borel encoding map $ p\mapsto \encfunc_p$ from $ P $ to $\prodspc{F} $ such that
			\begin{equation}\label{eq:encoding_disjoint}
				\forall p\neq r\in P, \ \Eq{D}(\encfunc_p,\encfunc_r)\in\idl{J}.
			\end{equation}
			\item A Borel set of safe-to-copy elements $ T \subseteq P $, a Borel owner map $ \own:T\to Q $, and a Borel map $ p\mapsto u_p $ from  $ T $ to $ \prodspc{F} $.
			\item A Borel map $ q\mapsto B(q)$ from $ Q $ to $ \prst(D)  $ such that
			\begin{equation}\label{eq:disjoint_S_domain}
				\forall q\neq r \in Q, \  B(q)\in\idl{J}^+  \text{ and } B(q)\cap B(r)\in\idl{J}.
			\end{equation}
			\item 
			\begin{itemize}
				\item[i.] For every $ p\in T $ we have that
				\begin{equation}\label{eq:safe_to_copy}
					\forall r\in P, \ B(\own(p)) \cap\Eq{D}(u_p,\encfunc_r)\in\idl{J}.
				\end{equation}
				\item[ii.] For any $ p\not=r\in T $ having the same owner i.e. $ \own(p) = \own(r) $, we have
				\begin{equation}\label{eq:same_owner}
					B(\own(p))\cap\Eq{D}(u_p,u_r)\in\idl{J}.
				\end{equation}
			\end{itemize}
			\item For every $ q\in Q $ and every $ i\in\omega $, there is $ p\in P $ such that at least one of the
			following holds.
			\begin{align}
				p\in T \ \wedge \ \own(p)=q \ &\wedge \ B(q)\subseteq_{\idl{J}}\Eq{D}(h_i^q,u_p).  \label{eq:catched_by_copy_owner}\\
				\neg[p\in T \ \wedge \ \own(p)=q] \ &\wedge \ 	B(q)\subseteq_{\idl{J}}\Eq{D}(h_i^q,\encfunc_p).  \label{eq:cathed_not_by_copy_owner}
			\end{align}
		\end{enumerate}
		
		Now, for $ p\in P $, define $ E_p\in\prodspc{F} $ by
		\begin{equation*}
			E_p(n)=
			\begin{cases}
				u_p(n),&p\in T\text{ and }
				n\in B(\own(p)),\\
				\encfunc_p(n),&\text{otherwise}.
			\end{cases}
		\end{equation*}
		Then $ \idl{E}:=\{E_p:p\in P\} $ is a Borel $ \idl{J} $-e.d. family such that for every $ q\in Q $ there are functions
		$ f_i\in\idl{E} $, $ i\in\omega $ with
		\begin{equation}\label{eq:omega_centered}
			\forall i\in \omega, \ B(q)\subseteq_{\idl{J}}\Eq{D}(h_i^q,f_i)
		\end{equation}
	\end{lem}
	\begin{proof}
		We first prove that the family is $ \idl{J} $-e.d.. For distinct $ E_p, E_r $, we have that $ p\not=r\in P $. So we will prove that for any $ p\neq r\in P $, $ \Eq{D}(E_p , E_r)\in\idl{J} $. Once it's proved, we have that the function $ p\mapsto E_p $ is injective.
		\begin{itemize}
			\item[(1)] Suppose $ p,r\notin T $. Then $ E_p=\encfunc_p $, $ E_r=\encfunc_r $, so $ \Eq{D}(E_p,E_r)\in\idl{J} $ by \eqref{eq:encoding_disjoint}.
			
			\item[(2)] Suppose $ p \in T$ but $ r\notin T $. We have that
			\[
			\Eq{D}(E_p,E_r) = \Eq{B(\own(p))}(u_p,\encfunc_r) \cup \Eq{B(\own(p))^\compl}(\encfunc_p,\encfunc_r)\in \idl{J}
			\]
			by \eqref{eq:safe_to_copy} and \eqref{eq:encoding_disjoint}. The case in which $ r\in T $ and
			$ p \notin T$ is symmetric.
			
			\item[(3)] Suppose $ p,r\in T $.
			\begin{itemize}
				\item[i.] If $ \own(p) = \own(r) $, by \eqref{eq:same_owner} we have $ B(\own(p))\cap\Eq{D}(u_p,u_r)\in\idl{J} $. Now
				\[
				\Eq{D}(E_p,E_r) = \Eq{B(\own(p))}(u_p,u_r) \cup \Eq{B(\own(p))^\compl}(\encfunc_p,\encfunc_r)\in \idl{J}.
				\]
				\item[ii.] If $ \own(p) \not= \own(r) $, we have
				\begin{align*}
					\Eq{D}(E_p,E_r) = &\Eq{ B( \own(p) ) \cap B( \own( r ) ) }(u_p,u_r)  & &\in\idl{J}\text{ by } \eqref{eq:disjoint_S_domain}\\
					\cup &\Eq{ B( \own(p) ) \setminus B( \own( r ) ) }(u_p,\encfunc_r)    & &\in\idl{J}\text{ by } \eqref{eq:safe_to_copy}\\
					\cup &\Eq{ B( \own(r) ) \setminus B( \own( p ) ) }(u_r,\encfunc_p)   & &\in\idl{J}\text{ by } \eqref{eq:safe_to_copy}\\
					\cup &\Eq{ ( B( \own(p) ) \cup B( \own( r ) ) )^\compl }(\encfunc_p,\encfunc_r)   & &\in\idl{J}\text{ by } \eqref{eq:encoding_disjoint}.
				\end{align*}
			\end{itemize}
		\end{itemize}
		This proves that $ \idl{E} $ is $ \idl{J} $-e.d..
		
		We next prove \eqref{eq:omega_centered}. Fix $ q\in Q $ and $ i\in\omega $. So by the assumption Lemma \ref{lem:protected_splicing} (6),
		\begin{itemize}
			\item[(1)] If \eqref{eq:catched_by_copy_owner} holds, note that $ E_p \uphar B( q ) = u_p \uphar B(q) $ where $ \own(p) = q $. So
			\[
			B(q)\setminus\Eq{D}(h_i^q,E_p)=B(q)\setminus\Eq{D}(h_i^q,u_p)\in\idl{J}.
			\]
			
			\item[(2)] If \eqref{eq:cathed_not_by_copy_owner} holds,
			\begin{itemize}
				\item[i.] If $ p\notin T $, then $ E_p = \encfunc_p $, so by the assumption itself we have $ B(q)\subseteq_{\idl{J}}\Eq{D}(h_i^q,E_p) $.
				\item[ii.] If $ p\in T $ and $ \own(p) \not= q $, similarly we have
				\begin{align*}
					&B(q)\setminus\Eq{D}(h_i^q,E_p)\\
					= \bigl( &B(q)\setminus\Eq{ B(\own(p))^\compl }(h_i^q,\encfunc_p)\bigr)  \cap \bigl(  \Eq{B(\own(p)) }(h_i^q,E_p) \bigr)^\compl \\
					\subseteq \phantom{\bigl(} &B(q)\setminus\Eq{ B(\own(p))^\compl }(h_i^q,\encfunc_p)  \\
					= \phantom{\bigl(} &B(q) \cap \bigl( B(\own(p))^\compl \cap \Eq{ D }(h_i^q,\encfunc_p) \bigr)^\compl  \\
					= \bigl( &B(q)\cap B(\own(p)) \bigr) \cup \bigl( B(q)\setminus \Eq{D}( h_i^q, \encfunc_p) \bigr) \in \idl{J} \quad \text{ by } \eqref{eq:disjoint_S_domain} \text{ and } \eqref{eq:cathed_not_by_copy_owner}.
				\end{align*} 
			\end{itemize}
		\end{itemize}
		
		For every $ i $, choose one output $ f_i=E_p $ obtained from above, which proves \eqref{eq:omega_centered}.
		
		Finally we prove that $ \idl{E} $ is Borel. For every $ n\in D, m< F(n) $, we have that
		\begin{align*}
			\{p\in P: E_p (n) = m\} = \biggl\{p\in P:  [p\notin T \ &\wedge \ \encfunc_p(n) = m ] \\
			\vee \Bigl[ p\in T \ &\wedge \ \bigl[ [n\in B(\own(p)) \ \wedge \ u_p(n) = m ] \\
			& \vee \ \phantom{\bigl[}[ n\notin B(\own(p)) \ \wedge  \  \encfunc_p(n)=m ]  \bigr]  \Bigr]  \biggr\},
		\end{align*}
		which is Borel, so the function $ p\mapsto  E_p$ is Borel. As it is also injective, by Lusin--Souslin theorem, the range $ \idl{E}=\{E_p:p\in P\} $ is Borel.
	\end{proof}
	
	We next give hypotheses which uniformly produce the data of Lemma
	\ref{lem:protected_splicing}. From now until the end of the construction, we assume $ \prodspc{F}=\bairesp $.
	
	\begin{defn}[{\cite{he2024maximal}}]\label{def:uniform_weak_properties}
		Let $ \idl{J} $ be a Borel ideal on $ \omega $.
		\begin{enumerate}[label=(\arabic*),font=\upshape]
			\item The ideal $ \idl{J} $ is called \textbf{uniformly weakly
				$ P^+ $}, if there is a Borel map $ \Psi $ defined on all $ \subseteq $-decreasing sequences $ \la X_n:n\in\omega\ra $ in $ \idl{J}^+ $ such that $ \Psi(\la X_n:n\in\omega\ra)\in\idl{J}^+ $ and $ \Psi(\la X_n:n\in\omega\ra)\subseteq_{\idl{J}}X_n $ for every $ n\in\omega $.
			\item A coloring $ c:[\omega]^2\to2 $ is called \textbf{subadditive}, if
			\[
			c(\{m,n\})\leq c(\{m,k\})+c(\{n,k\})
			\]
			whenever $ m<n<k $. The ideal $ \idl{J} $ is called \textbf{uniformly weakly Ramsey}, if there is a Borel map $ \Phi $ defined on all pairs $ (A,c) $ with $ A\in\idl{J}^+ $ and $ c $ subadditive, such that $ H:=\Phi(A,c) $ satisfies $ H\in(\idl{J}\uphar A)^+ $ and exactly one of the following holds.
			\begin{itemize}
				\item[i.] $ H $ is $ 0 $-homogeneous for $ c $.
				\item[ii.] no $ \idl{J}^+ $ subset of $ H $ is $ 0 $-homogeneous for $ c $.
			\end{itemize}
		\end{enumerate}
	\end{defn}
	
	\begin{prop}[{\cite[Lemma~5.2]{he2024maximal}}]\label{lem:uniform_p_to_ramsey}
		Every Borel uniformly weakly $ P^+ $ ideal on $ \omega $ containing $ \mathrm{Fin} $ is uniformly weakly Ramsey.
	\end{prop}

	\begin{prop}[{\cite[Lemma~5.3]{he2024maximal}}]\label{lem:standard_uniform_weak_p_plus}
	Every $ F_\sigma $ ideal on $ \omega $ containing $ \mathrm{Fin} $, and every countable iterated Fr\'echet ideal $ \finalpha{\alpha} $, $ 1\leq\alpha<\omega_1 $, is uniformly weakly $ P^+ $.
\end{prop}

We note that, it's a folklore that every $ F_\sigma $ ideal is $ P^+ $ \cite[Lemma~3.2.4]{meza2009ideals}, and then is weakly $ P^+ $. Based on the proof of \cite[Lemma~5.3]{he2024maximal}, we have the following.

\begin{prop}\label{prop:uniform_weak_p_plus_fubini}
	Let $ \idl{I} $ and $ \idl{K} $ be Borel ideals on $ \omega $, and suppose that $ \mathrm{Fin}\subseteq\idl{I} $. If $ \idl{I} $ is uniformly weakly $ P^+ $, then $ \idl{I}\times\idl{K} $ is uniformly weakly $ P^+ $.
\end{prop}
\begin{proof}
	Fix a Borel uniform weak-$ P^+ $ witness $ \Psi $ for $ \idl{I} $, and let $ \la A_n:n\in\omega\ra $ be a $ \subseteq $-decreasing sequence in $ (\idl{I}\times\idl{K})^+ $. For every $ n\in\omega $, define $ P_n=\{j\in\omega:A_n(j)\notin\idl{K}\} \in\idl{I}^+ $, which is also decreasing. Let $P=\Psi(\la P_n:n\in\omega\ra), \ P'=P\cap P_0  $. Then $ P'\in\idl{I}^+ $ and for all $ n\in\omega $ we have $ P'\subseteq_{\idl{I}}P_n $.
	
	For $ j\in P' $, as $ j\in P_0 $, we can define $ r(j)=\max\{k\leq j:j\in P_k\} $, and then define $ B\subseteq\omega\times\omega $ by
	\[
	B(j)=
	\begin{cases}
		A_{r(j)}(j),&j\in P',\\
		\emptyset,&j\notin P'.
	\end{cases}
	\]
	This is a Borel computation. We show that $ B $ is the witness of uniform weak $ P^+ $ property.
	
	For every $ j\in P' $, we have $ j\in P_{r(j)} $, so $ B(j)=A_{r(j)}(j)\notin\idl{K} $, which means $ \{j\in\omega:B(j)\notin\idl{K}\}=P'\in\idl{I}^+ $ and then $ B\in(\idl{I}\times\idl{K})^+ $.
	
	For any $ n\in\omega $, if a $ j\in P'\cap P_n $ has $ j\geq n $, then by definition $ r(j)\geq n $. Since the sequence $ \la A_k:k\in\omega\ra $ is decreasing, for this $ j $ we have $ B(j)=A_{r(j)}(j)\subseteq A_n(j) $. So
	\[
	\{j\in\omega:B(j)\setminus A_n(j)\notin\idl{K}\}\subseteq(P'\setminus P_n)\cup\{0,\ldots,n-1\}\in\idl{I},
	\]
	i.e. $ B\setminus A_n\in\idl{I}\times\idl{K} $.
\end{proof}

\begin{cor}\label{cor:uniform_weak_p_plus_fubini_closure}
	Every ideal in the smallest class of Borel ideals containing all $ F_\sigma $ ideals which contain $ \mathrm{Fin} $ and closed under Fubini products is uniformly weakly $ P^+ $.
\end{cor}
\begin{proof}
	By Proposition \ref{lem:standard_uniform_weak_p_plus}, every $ F_\sigma $ ideal containing $ \mathrm{Fin} $ is uniformly weakly $ P^+ $. The conclusion follows from Proposition \ref{prop:uniform_weak_p_plus_fubini}.
\end{proof}
	
	Fix a recursive injection $ \sharp: \fintree\to\omega $, and its partial inverse $ \natural $ such that $ \natural( \sharp (s)) = s $. For $ x\in\bairesp $, define $ \encfunc(x)(n):=\sharp(x\uphar n)$ for every $n\in\omega $. If $ x\neq y $, then $ x\uphar n\neq y\uphar n $ for all sufficiently
	large $ n $, so $ 	\Eq{\omega}(\encfunc(x),\encfunc(y))\in\mathrm{Fin}\subseteq\idl{J} $.
	
	\begin{lem}\label{lem:borel_positive_supports}
		Let $ Q $ be a standard Borel space, and let $ \idl{J} $ be a Borel ideal on $ \omega $ containing $ \mathrm{Fin} $. Then there is a Borel map $ q\mapsto A(q)$ from $ Q $ to $\idl{J}^+ $ such that $ A(q)\cap A(r)\in\idl{J} $ for any $ q\neq r \in Q$.
	\end{lem}
	\begin{proof}
		From the proof of \cite[Theorem 4.8]{he2024maximal}, we have a Borel injection $ C: \bairesp \to \idl{J}^+ $ such that the image of $ C $ is $ \idl{J} $-a.d.. We take another Borel injection $ \rho:Q\to\bairesp $. So the image of $ A := C \circ \rho : Q\to \idl{J}^+$ is also $ \idl{J} $-a.d..
	\end{proof}
	
	\begin{defn}[{\cite[Example~4.2]{he2024maximal}}]
		Fix $f\in\bairesp$. For each $m,n\in\omega$, we define $m\prec_f n$ if $m<n$, and $ f(m),f(n) \in \sharp'' \finfunc{\omega} $ such that 
		\begin{itemize}
			\item[(1)] $|\natural f(m)|=m$, $|\natural f(n)|=n$.
			\item[(2)] $\natural f(m) \subseteq \natural f(n)$.
		\end{itemize}
		
		Define
		\begin{equation*}
			c_f(\{m,n\})=%
			\begin{cases}
				0, &\text{if $m\prec_f n$,} \\
				1,  &\text{otherwise}.
			\end{cases}
		\end{equation*}
	\end{defn}
	Note that $c_f$ is subadditive.
	
	The following is implicitly used in \cite{he2024maximal}. Here we state it as a lemma.
	
	\begin{lem}\label{lem:prefix_dichotomy}
		Suppose that $ \idl{J} $ is Borel and uniformly weakly Ramsey such that $ \mathrm{Fin} \subseteq \idl{J}$. There are Borel maps which assign to every $ (A,h)\in\idl{J}^+\times\bairesp $ a set $ H(A,h)\in(\idl{J}\uphar A)^+ $, a number $ \tau(A,h)<2 $, and a real $ g(A,h)\in\bairesp $ such that
		\begin{enumerate}[label=(\arabic*),font=\upshape]
			\item If $ \tau(A,h)=0 $, then $ H(A,h)\subseteq\Eq{\omega}(h,\encfunc(g(A,h))) $.
			\item if $ \tau(A,h)=1 $, then for any $ g\in\bairesp $ we have $ \Eq{H(A,h)}(h,\encfunc(g))\in\idl{J} $.
		\end{enumerate}
		Consequently, we have that
		\begin{align*}
			&\{(A,h)\in \idl{J}^+ \times \bairesp: \tau(A,h) = 0\}\\
			= &\{(A,h)\in \idl{J}^+ \times \bairesp: \exists g\in\bairesp \ \Eq{H(A,h)}(h,\encfunc(g)) \in \idl{J}^+\}
		\end{align*}
		is Borel.
	\end{lem}
	\begin{proof}
		Once conditions (1), (2) and the Borelness of these functions are proved, we will automatically have the equality.
		
		Let $ \Phi $ witnesses uniform weak Ramseyness.
		
		For any $ (A,h) \in \idl{J}^+ \times \bairesp$,
		\begin{itemize}
			\item[(1)] Define $ H(A,h):=\Phi(A,c_h) $.
			\item[(2)] Define
			\begin{equation*}
				\tau(A,h):= 
				\begin{cases}
					0, &\text{if $H(A,h)$ is 0-homogeneous for $ c_h $}, \\
					1,  &\text{otherwise}.
				\end{cases}
			\end{equation*}
			\item[(3)] If $ \tau(A,h) =0$, by the definition of 0-homogeneity for $ c_h $, we have that for any $ n\in H(A,h) $, $ \natural h(n) $ exists. So we define
			\begin{equation*}
				g(A,h):= 
				\begin{cases}
					\bigcup_{n\in H(A,h)}\natural h(n), &\text{if $\tau(A,h) = 0$}, \\
					g_0,  &\text{otherwise}.
				\end{cases}
			\end{equation*}
			where $ g_0 $ is a fixed real.
		\end{itemize}

		Now, if $ \tau(A,h) = 0$, for any $ n\in H(A,h) $, we have
		\[
		\encfunc(g(A,h))(n) = \sharp (g(A,h) \uphar n) = \sharp \left(\left(\bigcup_{n\in H(A,h)}\natural h(n) \right) \uphar n \right) = h(n),
		\]
		so condition (1) is proved.
		
		If $ \tau(A,h) = 1 $, by its definition and the property of $ H(A,h) $, we have that any $ K \in (\idl{J}\uphar H(A,h) )^+ $ is not 0-homogeneous for $ c_h $. Note that
		\[
		\Eq{H(A,h)}(h,\encfunc(g)) = \{n\in H(A,h): h(n) = \sharp (g\uphar n)\}.
		\]
		So for any $ m<n \in  \Eq{H(A,h)}(h,\encfunc(g))$, we have $ \natural h(m) = g\uphar m $ and $ \natural h(n) = g\uphar n $. So $ m\prec_h n $ and thus $ \Eq{H(A,h)}(h,\encfunc(g))  $ is 0-homogeneous and is in $ \idl{J} $. So condition (2) is true.
		
		For Borelness, note that the function $ h\mapsto c_h $ and $ \Phi $ is Borel, so $ (A,h)  \mapsto H(A,h)$ is Borel. The function $ \tau(A,h) $is also Borel, as ``$ H $ is 0-homogeneous for $ c $'' is a Borel condition for $ H $ and $ c $. The Borelness of $ (A,h)\mapsto g(A,h) $ is from the Borelness of $ \tau $ and that for $ \tau(A,h) =0$, the value $ \left(\bigcup_{n\in H(A,h)}\natural h(n)\right)(i)=j $  if and only if
		\[
		\exists n\in H(A,h) \ \exists s \in \finfunc{\omega} \text{ with } |s|> i \ \wedge \ s(i)=j, \text{ such that } h(n) =\sharp s.
		\]
	\end{proof}
	
	\begin{lem}\label{lem:fresh_codes}
		Let $ Q $ be a standard Borel space with a Borel map $ q\mapsto\la h_j^q:j\in\omega\ra\in(\bairesp)^\omega $, and let $ R\subseteq Q\times\omega $ be Borel. There is a Borel injection $ \Gamma:R\to\bairesp $ such that for every $ (q,i)\in R $ we have
		\[
		\forall j\in \omega \ \Eq{\omega}(h_j^q,\encfunc(\Gamma(q,i)))\in\mathrm{Fin}.
		\]
	\end{lem}
	\begin{proof}
		Choose a Borel injection $ r:R\to\bairesp $. Let $ (l,k)\mapsto <l,k> $ be a recursive injection from $ \omega\times \omega \to \omega $.
		
		We recursively define $ \Gamma:R\to\bairesp $ as
		\[
		\begin{aligned}
			\forall n\in\omega, \ \Gamma(q,i)(n):= \min\biggl\{ &<r(q,i)(n),k> : k\in\omega \\
			\wedge \ \sharp\Bigl( (\Gamma(q,i)\uphar n) {}^\frown \langle &<r(q,i)(n),k> \rangle \Bigr) > \max\left\{ h_j^q (n+1): j<n+1 \right\} \biggr\}.
		\end{aligned}
		\]
		
		Note that the definition is well-defined because varying $ k $ will give infinitely distinct values. As $ r $ is injective, for any distinct $ (q,i) $ and $ (p,j) $ there is $ n\in\omega $ such that $ r(q,i)(n)\neq r(p,j)(n) $. Also, there are $ k_{q,i}, k_{p,j}\in\omega $ such that
		\[
		\Gamma(q,i)(n) = <r(q,i)(n),k_{q,i}>, \ \Gamma(p,j)(n) = <r(p,j)(n),k_{p,j}>
		\]
		for this $ n $. As $ <\cdot,\cdot> $ is also injective, we have that $ \Gamma(q,i)(n) \neq \Gamma(p,j)(n) $ and thus $ \Gamma $ is injective. By the definition it's also Borel.
		
		Now, fix $ (q,i)\in R $. For any $ j\in\omega $, we have that
		\[
		\Eq{\omega}(h_j^q,\encfunc(\Gamma(q,i))) = \{n\in\omega: h_j^q(n) = \sharp(\Gamma(q,i)\uphar n)\}.
		\]
		For any number $ n+1 > j $, by the definition of $ \Gamma(q,i) $ we have that $ \sharp(\Gamma(q,i)\uphar (n+1)) > h_j^q (n+1) $. So for this $ n+1 $ we have $ n+1 \notin \Eq{\omega}(h_j^q,\encfunc(\Gamma(q,i)))  $. This proves the lemma.
	\end{proof}
	
	\begin{thm}\label{thm:borel_omega_cp}
	Let $ \idl{J} $ be a Borel uniformly weakly $ P^+ $ ideal on $ \omega $ containing $ \mathrm{Fin} $. Then there is a Borel \omegaCP $ \idl{J} $-MED family in $ \bairesp $.
	\end{thm}
	\begin{proof}
	Fix a Borel uniform weak-$ P^+ $ selector $ \Psi $. By Proposition \ref{lem:uniform_p_to_ramsey}, $ \idl{J} $ is uniformly weakly Ramsey. Let $ Q=(\bairesp)^\omega $, and for $ q\in Q $ write $ q=\la h_i^q:i\in\omega\ra $. By Lemma \ref{lem:borel_positive_supports}, fix a Borel map $ q\mapsto A(q)\in\idl{J}^+ $ such that $ A(q)\cap A(r)\in\idl{J} $ for all distinct $ q,r\in Q $.
	
	For $ q=\la h_0^q,\ldots,h_{k}^q,\ldots \ra\in Q $, let $ X_0(q)=A(q) $, and recursively define $ X_{i+1}(q)\in\idl{J}^+ $, $ \tau_i(q)<2 $, and $ g_i^q\in\bairesp $ for $ i\in\omega $ by
	\[
	\forall i\in\omega,\ X_{i+1}(q)=H(X_i(q),h_i^q),\ \tau_i(q)=\tau(X_i(q),h_i^q)\text{ and }g_i^q=g(X_i(q),h_i^q).
	\]
	By Lemma \ref{lem:prefix_dichotomy}, for every $ i $ we have $ X_{i+1}(q)\subseteq X_i(q) $ and $ X_{i}(q)\in\idl{J}^+ $.
	
	Following the terminology in \cite{schrittesser2017horowitz}, say that $ i $ is \textbf{coherent} if $ \tau_i(q)=0 $, and \textbf{tangled} if $ \tau_i(q)=1 $.
	
	Define
	\[
	Z(q)=\Psi(\la X_i(q):i\in\omega\ra),\ X(q)=Z(q)\cap A(q).
	\]
	The map $ q\mapsto X(q) $ is Borel. Since $ Z(q)\in\idl{J}^+ $ and $ Z(q)\subseteq_{\idl{J}}X_0(q)=A(q) $, we have $ X(q)\in\idl{J}^+ $. Also, $ X(q)\subseteq A(q) $ and $ X(q)\subseteq_{\idl{J}}X_i(q) $ for every $ i\in\omega $.
	
	If $ i $ is coherent, then $ X_{i+1}(q)\subseteq\Eq{\omega}(h_i^q,\encfunc(g_i^q)) $, and then
	\begin{equation}\label{eq:omega_coherent_contained_h_e}
		X(q)\subseteq_{\idl{J}}\Eq{\omega}(h_i^q,\encfunc(g_i^q)).
	\end{equation}
	If $ i $ is tangled, then
	\begin{equation}\label{eq:omega_free_h_e_small}
		\forall x\in\bairesp,\  \Eq{X(q)}(h_i^q,\encfunc(x))\subseteq\bigl(X(q)\setminus X_{i+1}(q)\bigr)\cup\Eq{X_{i+1}(q)}(h_i^q,\encfunc(x))\in\idl{J}.
	\end{equation}
	
	Fix an enumeration $ \la(i_s,j_s):s\in\omega\ra $ of $ \{(i,j)\in\omega^2:i<j\} $. Starting with $ Y_0(q)=X(q) $, define a $ \subseteq $-decreasing sequence $ \la Y_s(q):s\in\omega\ra $ in $ \idl{J}^+ $. Suppose $ Y_s(q)\in\idl{J}^+ $ has been defined. We define
	\[
	Y_{s+1}(q)=
	\begin{cases}
		Y_s(q),&\text{at least one of $ i_s,j_s $ is not tangled},\\
		\Eq{Y_s(q)}(h_{i_s}^q,h_{j_s}^q),&\text{$ i_s,j_s $ are tangled and } \Eq{Y_s(q)}(h_{i_s}^q,h_{j_s}^q)\in\idl{J}^+,\\
		Y_s(q)\setminus\Eq{\omega}(h_{i_s}^q,h_{j_s}^q),&\text{$ i_s,j_s $ are tangled and } \Eq{Y_s(q)}(h_{i_s}^q,h_{j_s}^q)\in\idl{J}.
	\end{cases}
	\]
	Note that in the third case we have $ Y_{s+1}(q)\in\idl{J}^+ $. Also, for any $ s $ the map $ q\mapsto Y_s(q) $ is Borel.
	
	Again, define
	\[
	W(q)=\Psi(\la Y_s(q):s\in\omega\ra),\qquad B(q)=W(q)\cap Y_0(q).
	\]
	Since $ W(q)\in\idl{J}^+ $ and $ W(q)\subseteq_{\idl{J}}Y_0(q) $, we have $ B(q)\in\idl{J}^+ $. Also, $ B(q)\subseteq X(q)\subseteq A(q) $ and $ B(q)\subseteq_{\idl{J}}Y_s(q) $ for every $ s\in\omega $.
	
	\begin{claim}\label{claim:omega_B_q_dichotomy}
		Let $ i<j $ be a tangled pair such that $ (i,j)=(i_s,j_s) $ for some $ s\in\omega $. Then either $ B(q)\subseteq_{\idl{J}}\Eq{\omega}(h_i^q,h_j^q) $ or $ B(q)\cap\Eq{\omega}(h_i^q,h_j^q)\in\idl{J} $.
	\end{claim}
	\begin{claimproof}{\ref{claim:omega_B_q_dichotomy}}
		Suppose $ B(q)\not\subseteq_{\idl{J}}\Eq{\omega}(h_i^q,h_j^q) $. Since $ B(q)\subseteq_{\idl{J}}Y_{s+1}(q) $, if $ Y_{s+1}(q)=\Eq{Y_s(q)}(h_i^q,h_j^q)$, we have
		\[
		B(q)\setminus\Eq{\omega}(h_i^q,h_j^q)\subseteq B(q)\setminus Y_{s+1}(q)\in\idl{J},
		\]
		a contradiction. So $
		Y_{s+1}(q)=Y_s(q)\setminus\Eq{\omega}(h_i^q,h_j^q)$, which implies $ \Eq{\omega}(h_i^q,h_j^q) \subseteq Y_{s+1}(q)^\compl $. Then we have
		\[
		B(q)\cap\Eq{\omega}(h_i^q,h_j^q)\subseteq B(q)\setminus Y_{s+1}(q)\in\idl{J}.
		\]
	\end{claimproof}
	
	Now, for tangled $ i,j\in\omega $, define $ i\sim_q j $ if $ B(q)\subseteq_{\idl{J}}\Eq{\omega}(h_i^q,h_j^q) $. This is an equivalence relation. Indeed, reflexivity and symmetry are immediate. If $ i\sim_q j $ and $ j\sim_q l $, then
	\[
	B(q)\setminus\Eq{\omega}(h_i^q,h_l^q)\subseteq\bigl(B(q)\setminus\Eq{\omega}(h_i^q,h_j^q)\bigr)\cup\bigl(B(q)\setminus\Eq{\omega}(h_j^q,h_l^q)\bigr)\in\idl{J},
	\]
	so $ i\sim_q l $. Let
	\[
	R=\{(q,i)\in Q\times\omega:i\text{ is $q$-tangled and }\forall\text{ $q$-tangled }i'<i,\ i'\not\sim_q i\},
	\]
	which is Borel. That is, for every $ q $, the set $ R(q) $ contains the least numbers of all $ \sim_q $-classes of tangled indices.
	
	Apply Lemma \ref{lem:fresh_codes} to $ R $ and the identity map $ q\mapsto\la h_j^q:j\in\omega\ra $. Let $ \Gamma:R\to\bairesp $ be the resulting Borel injection. Thus, for every $ (q,i)\in R $ and every $ j\in\omega $, we have that $ \Eq{\omega}(h_j^q,\encfunc(\Gamma(q,i)))\in\mathrm{Fin} $.
	
	Let $ P=\bairesp $ and $ T=\Gamma''R $. By the Lusin--Souslin theorem, $ T $ is Borel and $ \Gamma^{-1}:T\to R $ is Borel. For any $ x\in T $, there is unique $ (q,i) $ such that $ \Gamma^{-1}(x)=(q,i) $, and then define $ \own(x)=q $ and $ u_x=h_i^q $, which are Borel. For every $ x\in P $, define $ \encfunc_x=\encfunc(x) $, and for every $ q\in Q $, Define $ B(q) $ as above.
	
	We verify the assumptions of Lemma \ref{lem:protected_splicing}.
	\begin{itemize}
		\item[(1)] The map $ q\mapsto\la h_i^q:i\in\omega\ra $ is Borel.
		
		\item[(2)] The map $ x\mapsto\encfunc_x $ is Borel. If $ x\neq y $, then $ \Eq{\omega}(\encfunc_x,\encfunc_y)\in\mathrm{Fin}\subseteq\idl{J} $.
		
		\item[(3)] The set $ T $ and the maps $ \own $ and $ x\mapsto u_x $ are Borel.
		
		\item[(4)] For every $ q\in Q $, we have $ B(q)\in\idl{J}^+ $. If $ q\neq r $, then $ B(q)\cap B(r)\subseteq A(q)\cap A(r)\in\idl{J} $.
		
		\item[(5)] Let $ x\in T $, and write $ \Gamma^{-1}(x)=(q,i) $.
		\begin{itemize}
			\item[i.] Note that $ i $ is tangled and $ u_x=h_i^q $. Since $ B(q)\subseteq X(q) $, by \eqref{eq:omega_free_h_e_small}, for every $ y\in P $ we have
			\[
			B(\own(x))\cap\Eq{\omega}(u_x,\encfunc_y)\subseteq\Eq{X(q)}(h_i^q,\encfunc(y))\in\idl{J}.
			\]
			Thus \eqref{eq:safe_to_copy} holds.
			
			\item[ii.] 	Suppose $ x\neq y\in T $ and $ \own(x)=\own(y)=q $. Write $ \Gamma^{-1}(x)=(q,i) $ and $ \Gamma^{-1}(y)=(q,j) $. Since $ \Gamma $ is injective, $ i\neq j $. Also, $ i,j\in R(q) $, so they belong to distinct $ \sim_q $-classes. By Claim \ref{claim:omega_B_q_dichotomy},
			\[
			B(\own(x))\cap\Eq{\omega}(u_x,u_y)=B(q)\cap\Eq{\omega}(h_i^q,h_j^q)\in\idl{J}.
			\]
			Thus \eqref{eq:same_owner} holds.
		\end{itemize}
		
		\item[(6)] Fix $ q=\la h_i^q:i\in\omega\ra\in Q $ and $ t\in\omega $.
		\begin{itemize}
			\item[i.] Suppose $ t $ is $ q $-tangled. Let $ i $ be the least in the $ \sim_q $-class of $ t $, and let $ x=\Gamma(q,i) $. Then $ x\in T $, $ \own(x)=q $, and
			\[
			B(q)\subseteq_{\idl{J}}\Eq{\omega}(h_t^q,h_i^q)=\Eq{\omega}(h_t^q,u_x).
			\]
			Thus \eqref{eq:catched_by_copy_owner} holds.
			
			\item[ii.] Suppose $ t $ is $ q $-coherent, and let $ x=g_t^q $. Since $ B(q)\subseteq X(q) $, by \eqref{eq:omega_coherent_contained_h_e} we have $ B(q)\subseteq_{\idl{J}}\Eq{\omega}(h_t^q,\encfunc_x) $.  We claim that $ x\notin T \ \vee \ \own(x)\neq q $. Otherwise, let $ x=\Gamma(q,i) $, and by Lemma \ref{lem:fresh_codes} we have $ \Eq{\omega}(h_t^q,\encfunc_x)\in\mathrm{Fin}\subseteq\idl{J} $, which implies $ B(q)\in J $, a contradiction. Thus \eqref{eq:cathed_not_by_copy_owner} holds.
		\end{itemize}
	\end{itemize}
	
	Thus all the assumptions of Lemma \ref{lem:protected_splicing} hold. Let $ \idl{E}\subseteq\bairesp $ be the Borel $ \idl{J} $-e.d. family given by the lemma. For every $ q=\la h_i^q:i\in\omega\ra\in Q $, there are $ f_i\in\idl{E} $ such that $ B(q)\subseteq_{\idl{J}}\Eq{\omega}(h_i^q,f_i) $ for every $ i\in\omega $. Since $ B(q)\in\idl{J}^+ $, the family $ \idl{E} $ is \omegaCP $ \idl{J} $-MED.
\end{proof}
	
	\begin{rem}\label{rem:borel_fp_uniform_weak_ramsey}
	The construction in Theorem \ref{thm:borel_omega_cp} has a finite version which requires only uniform weak Ramseyness. That is, if $ \idl{J} $ is a Borel uniformly weakly Ramsey ideal on $ \omega $ containing $ \mathrm{Fin} $, then there is a Borel \FP $ \idl{J} $-MED family in $ \bairesp $.
	
	Indeed, for every $ 1\leq k<\omega $ let $ Q_k=(\bairesp)^k $, and let $ Q=\bigsqcup_{1\leq k<\omega}Q_k $. For $ q=\la h_0^q,\ldots,h_{k-1}^q\ra\in Q_k $, define $ \la\widehat{h_i^q}:i\in\omega\ra $ by
	\[
	\widehat{h_i^q}=
	\begin{cases}
		h_i^q,&i<k,\\
		h_0^q,&i\geq k.
	\end{cases}
	\]
	By Lemma \ref{lem:borel_positive_supports}, choose Borel sets $ A(q)\in\idl{J}^+ $ such that $ A(q)\cap A(r)\in\idl{J} $ for any $ q\neq r $.
	
	For $ q\in Q_k $, start with $ X_0(q)=A(q) $ and run the recursion using Lemma \ref{lem:prefix_dichotomy} only for $ i<k $ and let $ X(q)=X_k(q) $. As in the countable construction, say that $ i<k $ is coherent if $ \tau_i(q)=0 $ and tangled if $ \tau_i(q)=1 $. Fix an enumeration $ \la(i_s^k,j_s^k):s<N_k\ra $ of $ \{(i,j):i<j<k\} $, where $ N_k=\binom{k}{2} $, and starting with $ Y_0(q)=X(q) $, run the same three-case recursion as in the proof of Theorem \ref{thm:borel_omega_cp}. Similarly, we also define $ B(q)=Y_{N_k}(q)$. Then $ B(q)\in\idl{J}^+ $, $ B(q)\subseteq X(q)\subseteq A(q) $, and for every tangled pair $ i<j<k $ exactly one of the following occurs:
	\[
	B(q)\subseteq\Eq{\omega}(h_i^q,h_j^q) \ \text{ or } \ B(q)\cap\Eq{\omega}(h_i^q,h_j^q)=\emptyset.
	\]
	
	For tangled $ i,j<k $, define $ i\sim_q j $ if $ B(q)\subseteq\Eq{\omega}(h_i^q,h_j^q) $, and similarly let
	\[
	R=\{(q,i)\in Q\times\omega:i<|q|,\ i\text{ is $ q $-tangled, and }(\forall\text{ $ q $-tangled }i'<i)\ i'\not\sim_q i\}.
	\]
	Apply Lemma \ref{lem:fresh_codes} to $ R $ and the Borel map $ q\mapsto\la\widehat{h_i^q}:i\in\omega\ra $. Define $ P $, $ T $, $ \own $, $ u_x $, and the background codes exactly as in the proof of Theorem \ref{thm:borel_omega_cp}. The verification of the hypotheses of Lemma \ref{lem:protected_splicing} is the same, using the exact dichotomy above similar to Claim \ref{claim:omega_B_q_dichotomy}; for the case $ i\geq k $ in assumption of Lemma \ref{lem:protected_splicing} (6), one reuses the witness in the case $ i=0 $, since $ \widehat{h_i^q}=h_0^q $.
	
	So Lemma \ref{lem:protected_splicing} gives a Borel $ \idl{J} $-e.d. family $ \idl{E}\subseteq\bairesp $ such that \eqref{eq:omega_centered} further gives \FP $ \idl{J} $-MEDness.
	
	In particular, for $ \idl{J}=\mathrm{Fin} $ this gives a Borel finitely productive MED family as in \cite[Theorem~4.7]{schembecker2026productivity}.
\end{rem}

\begin{cor}\label{cor:borel_productive_consequences}
	Let $ \idl{J} $ be a Borel ideal on $ \omega $ containing $ \mathrm{Fin} $.
	\begin{enumerate}[label=(\arabic*),font=\upshape]
		\item If $ \idl{J} $ is uniformly weakly Ramsey, then there is a Borel \FP $ \idl{J} $-MED family in $ \bairesp $.
		\item If $ \idl{J} $ is uniformly weakly $ P^+ $, then there is a Borel \omegaCP $ \idl{J} $-MED family in $ \bairesp $, and then there is a Borel \FP $ \idl{J} $-MED family in $ \bairesp $.
	\end{enumerate}
\end{cor}
\begin{proof}
	(1) follows from Remark \ref{rem:borel_fp_uniform_weak_ramsey}. For (2), apply Theorem \ref{thm:borel_omega_cp} and Proposition \ref{prop:elementary_productive_implications}.
\end{proof}

\begin{cor}\label{cor:borel_omega_cp_fsigma_frechet}
	Let $ \idl{J} $ be a Borel ideal on $ \omega $ satisfying at least one of the following conditions.
	\begin{enumerate}[label=(\arabic*),font=\upshape]
		\item The ideal $ \idl{J} $ is a countable iterated Fr\'echet ideal $ \finalpha{\alpha} $ for some $ 1\leq\alpha<\omega_1 $.
		\item The ideal $ \idl{J} $ belongs to the smallest class of Borel ideals containing all $ F_\sigma $ ideals which contain $ \mathrm{Fin} $ and closed under Fubini products.
	\end{enumerate}
	Then there is a Borel \omegaCP $ \idl{J} $-MED family in $ \bairesp $, and then there is a Borel \FP $ \idl{J} $-MED family in $ \bairesp $.
\end{cor}
\begin{proof}
	In case (1), the ideal $ \idl{J} $ is uniformly weakly $ P^+ $ by Proposition \ref{lem:standard_uniform_weak_p_plus}. In case (2), the ideal $ \idl{J} $ is uniformly weakly $ P^+ $ by Corollary \ref{cor:uniform_weak_p_plus_fubini_closure}. The conclusion follows from Corollary \ref{cor:borel_productive_consequences}.
\end{proof}

	\subsection{No analytic \FSP $ \mathrm{Fin} $-MED family}

\begin{lem}\label{lem:fsp_global_ultrafilter}
	Let $ F $ be a partial function, $ D=\dom(F) $, $ \idl{J} $ an ideal on $ D $, and let $ \idl{E}\subseteq\prodspc{F} $ be a $ \idl{J} $-e.d. family. The following are equivalent.
	\begin{enumerate}[label=(\arabic*),font=\upshape]
		\item The family $ \idl{E} $ is \FSP $ \idl{J} $-MED.
		\item There is an ultrafilter $ \U $ on $ D $ such that $ \idl{J}^*\subseteq\U $ and $ \idl{E} $ is $ \U^* $-MED.
	\end{enumerate}
\end{lem}
\begin{proof}
	\phantom{a}
	
	\begin{itemize}
		\item[(1)] Suppose $ M:\prodspc{F}\to\finsub{\idl{E}}\setminus\{\emptyset\} $ witnesses that $ \idl{E} $ is \FSP. Consider the compact product space $ K=\prod_{h\in\prodspc{F}}M(h) $, where every term has discrete topology. For every finite $ H\subseteq\prodspc{F} $, define
		\[
		K_H=\left\{s\in K:\bigcap_{h\in H}\Eq{D}(h,s(h))\in\idl{J}^+\right\},
		\]
		which is clopen, and by \FSP is non-empty. Also, for finite $ H_0,\ldots,H_{m-1}\subseteq\prodspc{F} $, we have $ K_{\bigcup_{i<m}H_i}\subseteq\bigcap_{i<m}K_{H_i} $. Thus $ \{K_H:H\in\finsub{\prodspc{F}}\} $ has the finite intersection property, and by compactness there is $ s\in K $ such that
		\[
		\forall H\in\finsub{\prodspc{F}},\ \bigcap_{h\in H}\Eq{D}(h,s(h))\in\idl{J}^+.
		\]
		
		We claim that $ \idl{J}^*\cup\{\Eq{D}(h,s(h)):h\in\prodspc{F}\} $ has the finite intersection property. Let $ h_0,\ldots,h_{k-1}\in\prodspc{F} $, $ J_0,\ldots,J_{l-1}\in\idl{J} $, and define $ A=\bigcap_{i<k}\Eq{D}(h_i,s(h_i))\in\idl{J}^+ $. Since $ \bigcup_{j<l}J_j\in\idl{J} $, we have $ A\setminus\bigcup_{j<l}J_j\in\idl{J}^+ $ i.e. $ A\cap\bigcap_{j<l}(D\setminus J_j)\neq\emptyset$.
		
		So it generates a proper filter, and we extend it to an ultrafilter $ \U $ on $ D $. Then $ \idl{J}^*\subseteq\U $ and for every $ h\in\prodspc{F} $ we have $ \Eq{D}(h,s(h))\in\U $. Since $ \idl{J}\subseteq\U^* $, the family $ \idl{E} $ is $ \U^* $-e.d., and since for every $ h\in\prodspc{F} $ we have $ \Eq{D}(h,s(h))\in\U=(\U^*)^+ $, the family $ \idl{E} $ is $ \U^* $-MED.
		
		\item[(2)] Conversely, suppose $ \U $ is an ultrafilter on $ D $ such that $ \idl{J}^*\subseteq\U $ and $ \idl{E} $ is $ \U^* $-MED. For every $ h\in\prodspc{F} $, choose $ f_h\in\idl{E} $ such that $ \Eq{D}(h,f_h)\in\U $, and define $ M(h)=\{f_h\} $. As for every $ h_0,\ldots,h_{k-1}\in\prodspc{F} $ we have $ \bigcap_{i<k}\Eq{D}(h_i,f_{h_i})\in\U $. Also, since $ \idl{J}^*\subseteq\U $, we have $ \U\subseteq \idl{J}^+ $. Thus $ M $ witnesses that $ \idl{E} $ is \FSP $ \idl{J} $-MED.
	\end{itemize}
\end{proof}

\begin{defn}
	A non-principal ultrafilter $ \U $ on $ \omega $ is called a \textbf{$ P $-point}, if for any decreasing sequence $ U_0\supseteq U_1\supseteq\cdots $ in $ \U $ there is a $ B\in\U $ such that $ \forall n\in\omega, \  B\subseteq^*U_n $.
\end{defn}

We use the standard equivalent partition characterization of $ P $-points; e.g. see \cite[Section~5.4]{goldbring2022ultrafilters}.

	\begin{prop}
	For a non-principal ultrafilter $ \U $ on $ \omega $, it's a $ P $-point if and only if any one of the following holds.
	\begin{enumerate}[label=(\arabic*),font=\upshape]
		\item For every partition $ \omega=\bigsqcup_{n\in\omega}P_n $ such that $ \forall n\in\omega, \  P_n\notin\U $, there is an $ X\in\U $ such that for any $ n\in\omega $ the intersection $ X\cap P_n $ is finite.
		\item For every function $ f:\omega\to\omega $, there is an $ X\in\U $ such that $ f\uphar X $ is either finite-to-one or constant.
	\end{enumerate}
\end{prop}

	\begin{defn}[{\cite{millan2020filters}}]
		Let $ q\in\bairesp$. A subset $ X\in\ramseysp $ is called \textbf{$ q $-rare}, if $ \forall m<n\in X,\ q(m)<n$.
	\end{defn}
	
	Also, given $ H\in\ramseysp $, define $ [H]^q := \{X\in \infsub{H}: X \text{ is } q\text{-rare}\} $.
	
	\begin{thm}[Taylor; see {\cite[Theorem~4(c)]{blass1988selective}}]\label{thm:p_point_sparse_homogeneity}
		Let $ \U $ be a $ P $-point on $ \omega $ and $ \idl{X}\subseteq\infsub{\omega} $ be analytic. Then there are $ H\in\U $ and $ q:\omega\to\omega $ such that either $ [H]^q \subseteq \idl{X} $ or $ [H]^q \cap \idl{X} = \emptyset $.
	\end{thm}

	Note that by letting $ q'(n): = n+1+ \max\{q(k):k\leq n\} $, we can assume that the $ q $ is strictly increasing and $ \forall n\in\omega, q(n)>n $, as in this case for any $ H\in\ramseysp $ we have $ [H]^{q'}\subseteq [H]^q $.
	
	The proof of the following theorem is also motivated by the trace argument in \cite[Section~3]{raghavan2010van}.

\begin{thm}\label{thm:no_analytic_fsp_fin}
	There is no analytic \FSP $ \mathrm{Fin} $-MED family in $ \bairesp $.
\end{thm}
\begin{proof}
	If not, let $ \idl{E}\subseteq\bairesp $ be analytic and \FSP $ \mathrm{Fin} $-MED. By Lemma \ref{lem:fsp_global_ultrafilter}, there is a non-principal ultrafilter $ \U $ on $ \omega $ such that $ \idl{E} $ is $ \U^* $-MED. Note that for any ideal $ \idl{J} $ which we assumed proper, a finite $ \idl{J} $-e.d. family in $ \bairesp $ will not be $ \idl{J} $-MED.
	
	\begin{claim}\label{claim:u_p_point}
		The ultrafilter $ \U $ is a $ P $-point.
	\end{claim}
	\begin{claimproof}{\ref{claim:u_p_point}}
		By the partition characterization above, let $ \omega=\bigsqcup_{n\in\omega}P_n $ with $ P_n\notin\U $ for every $ n\in\omega $.  Choose distinct $ f_n\in\idl{E} $, and define $ h\in\bairesp $ by $\forall n\in\omega, \ h\uphar P_n=f_n\uphar P_n $. Since $ \idl{E} $ is $ \U^* $-MED, choose $ f\in\idl{E} $ such that $ \Eq{\omega}(h,f)\in\U $. There is at most one $ n $ such that $ f=f_{n} $. If such $ n $ exists, let $ B=\Eq{\omega}(h,f)\setminus P_{n} $; otherwise let $ B=\Eq{\omega}(h,f) $. Then $ B\in\U $. If $ n $ exists, for every $ m\neq n$,
		\[
		B\cap P_m\subseteq\Eq{\omega}(f_m,f),
		\]
		which is finite, and $ B\cap P_{n} $ is empty. If $ n $ does not exist, then $ \forall n\in\omega, \ B\cap P_n $ is finite. Thus $ \U $ is a $ P $-point.
	\end{claimproof}
	
	Fix distinct $ f_n\in\idl{E} $, $ n\in\omega $. For $ x\in\ramseysp $, define $ p_i^x=x(2i),\ r_i^x=x(2i+1) $, and
	\[
	\Phi(x)=\{(r_i^x,f_{p_i^x}(r_i^x)):i\in\omega\}.
	\]
	Let
	\[
	\idl{P}=\{x\in\infsub{\omega}:\exists f\in\idl{E}\ \Phi(x)\subseteq\gr(f)\},
	\]
	which is analytic.
	
	By Claim \ref{claim:u_p_point} and Theorem \ref{thm:p_point_sparse_homogeneity}, choose $ H\in\U $ and an strictly increasing $ q:\omega\to\omega $ with $ q(n)>n $ for every $ n\in\omega $ such that either $ [H]^q \subseteq \idl{P} $ or $ [H]^q \cap \idl{P} = \emptyset $.
	
	\begin{claim}\label{claim:fsp_direct_hiding}
		There is a $ y\in [H]^q $ such that $ y\in\idl{P} $.
	\end{claim}
	\begin{claimproof}{\ref{claim:fsp_direct_hiding}}
		Let
		\[
		H'=\{m\in H:\exists n\in H, \ q(n)<m\}.
		\]
		The set $ H' $ contains a cofinite subset of $ H $, so $ H'\in\U $. For $ m\in H' $, define
		\[
		\ell(m)=\max\{n\in H:q(n)<m\}.
		\]
		Since $ q(n)>n $, we have $ \ell(m)< q(\ell(m))<m $. Also, $ \lim_{m\in H'}\ell(m)\to\infty $. This is because $ m < m_1 \in H' $ implies $ \ell(m)\leq \ell(m_1) $, and for any $ m\in H' $, if $ n $ is the smallest in $ H $ larger than $ \ell(m) $, then for any $ m_1 \in H' \cap (m,\infty) $ satisfying $ q(n)< m_1 $, we have $ \ell(m)< \ell(m_1) $.
		
		Define $ h\in\bairesp $ by $ h(m)=f_{\ell(m)}(m) $ for $ m\in H' $, and define it arbitrarily outside $ H' $. Since $ \idl{E} $ is $ \U^* $-MED, there is $ f\in\idl{E} $ such that $ \Eq{\omega}(h,f)\in\U $. Recursively choose $ m_0<m_1<\cdots $ in $ H' \cap \Eq{\omega}(h,f) $ such that $ \forall k\in \omega, \ q(m_k)< \ell(m_{k+1})$. So we have $ \forall k\in\omega, \ \ell(m_k)<m_k < q(m_k) < \ell(m_{k+1})$, and then let $ y=\{\ell(m_0)<m_0<\ell(m_1)<m_1<\cdots\} $.
		
		By the definition of $ \ell $, we have $ q(\ell(m_k))<m_k $. Also, by the definition of $ y $, we have $ q(m_k)<\ell(m_{k+1}) $. Thus $ y $ is $ q $-rare.
		
		Now, for every $ k\in\omega $, we have $ \Phi(y)(m_k) =  f_{\ell(m_k)}(m_k)=h(m_k)=f(m_k)$, so $ \Phi(y)\subseteq\gr(f) $. Thus $ y\in\idl{P} $.
	\end{claimproof}
	
	\begin{claim}\label{claim:fsp_direct_escape}
		There is a $ z\in [H]^q $ such that $ z\notin\idl{P} $.
	\end{claim}
	\begin{claimproof}{\ref{claim:fsp_direct_escape}}
		Let $ y=\{p_0<r_0<p_1<r_1<\cdots\} $ and $ f\in\idl{E} $ be as in Claim \ref{claim:fsp_direct_hiding}, so that $ \forall k\in\omega, \ f_{p_k}(r_k)=f(r_k) $. Let $ a\in H $ such that $ f_a\neq f $. We have that $ \Eq{\omega}(f_a,f) $ is finite, so let $ b\in H $ such that $ q(a)< b $ and $ f_a(b)\neq f(b) $. Now, let $ k_0\in\omega $ be the smallest such that $q(b)< p_{k_0} $, and let
		\[
		z=\{a<b<p_{k_0}<r_{k_0}<p_{k_0+1}<r_{k_0+1}<\cdots\}.
		\]
		We have that $ z $ is also $ q $-rare.
		
		Suppose $ z\in\idl{P} $, so there is $ g\in\idl{E} $ such that $ \Phi(z)\subseteq\gr(g) $. We have
		\[
		\begin{cases}
			\Phi(z)(\mathmakebox[\widthof{$b_k$}][c]{b}) = f_a(b),\\
			\Phi(z)(r_k) = f_{p_k}(r_k), \forall k \geq k_0.
		\end{cases}
		\]
		Thus $ \Eq{\omega}(f,g) $ is infinite, which means $ g=f $. But by assumption $ g(b) = \Phi(z)(b) = f_a(b) \neq f(b) $, a contradiction. Thus $ z\notin\idl{P} $.
	\end{claimproof}
	
	Claims \ref{claim:fsp_direct_hiding} and \ref{claim:fsp_direct_escape} give $ q $-rare infinite subsets of $ H $ with opposite $ \idl{P} $-colors, contradicting the choice of $ H $ and $ q $.
\end{proof}

\begin{cor}
	There is no non-principal ultrafilter $ \U $ on $ \omega $ such that there is an analytic subset $\idl{E}\subseteq \bairesp $ which is eventually different and $ \U^* $-maximal.
\end{cor}
\begin{proof}
	If such $ \U $ and $ \idl{E} $ exist, then by Lemma \ref{lem:fsp_global_ultrafilter} the family $ \idl{E} $ is analytic \FSP $ \mathrm{Fin} $-MED, a contradiction to Theorem \ref{thm:no_analytic_fsp_fin}.
\end{proof}

\section{Ultrafilter Maximality}

\subsection{A Stone space hierarchy}

Throughout this section, let $ F $ be a partial function, let $ D=\dom(F) $, let $ \idl{J} $ be an ideal on $ D $, and let $ \idl{E}\subseteq\prodspc{F} $ be $ \idl{J} $-e.d.. Lemma \ref{lem:fsp_global_ultrafilter} shows that \FSP maximality is exactly maximality with respect to one ultrafilter extending $ \idl{J}^* $. We now place this together with ordinary maximality, finite productivity, $ \omega $-centered productivity and \VD maximality in a common framework.

Let $ \beta D $ be the Stone space of ultrafilters over $ D $, and define
\begin{equation*}
	\beta_\idl{J}D :=\{\U\in \beta D:\idl{J}^*\subseteq\U\}\\
	=\bigcap_{A\in \idl{J}^*}\{\U\in \beta D: A\in \U\},
\end{equation*}
which is a non-empty compact zero-dimensional Hausdorff space as a subspace of $ \beta D $. For $ A\subseteq D $, define $ \what{A}:= \{\U\in \beta_\idl{J} D: A\in\U\}$, which is a clopen subset of $ \beta_\idl{J} D $.

 For $ h\in\prodspc{F} $, define
\[
O_h(\idl{E})=\{\U\in\beta_\idl{J}D:\exists f\in\idl{E}\ \Eq{D}(h,f)\in\U\}=\bigcup_{f\in\idl{E}}\what{\Eq{D}(h,f)},
\]
which is open.

	\begin{fact}\label{fact:stone_basic}
	The following hold.
	\begin{enumerate}[label=(\arabic*),font=\upshape]
		\item For every $ A\subseteq D $, $ 	\what{A}\neq\emptyset $ if and only if $ A\in\idl{J}^+ $.
		
		\item For every $ 1\leq k<\omega $ and $ A_0,\ldots,A_{k-1}\subseteq D $, $ \bigcap_{i<k}\what{A_i}=\what{\bigcap_{i<k}A_i} $.
		
		\item For every $ X\subseteq D $ and $ h\in\prodspc{F} $, 
		\[ \what{X}\cap O_h(\idl{E})=\bigcup_{f\in\idl{E}}\what{X\cap\Eq{D}(h,f)}. \] 
		
		\item For every $ 1\leq k<\omega $ and $ h_0,\ldots,h_{k-1}\in\prodspc{F} $,
		\[
		\bigcap_{i<k}O_{h_i}(\idl{E})=\bigcup_{\la f_0,\ldots,f_{k-1}\ra\in\idl{E}^k}\what{\bigcap_{i<k}\Eq{D}(h_i,f_i)}.
		\]
		
		\item For every fixed $ h\in\prodspc{F} $, the family $ \{\what{\Eq{D}(h,f)} $: $ f\in\idl{E}\} $ is pairwise disjoint.
	\end{enumerate}
\end{fact}
\begin{proof}
	\phantom{a}
	
	\begin{enumerate}[label=(\arabic*),font=\upshape]
		\item By $ \what{A}\neq \emptyset $, there is $ \U \supseteq \idl{J}^* $ with $ A\in\U $. We have that $ \U \cap \idl{J} = \emptyset$, so $ A\in \idl{J}^+ $.  Conversely, if $ A\in\idl{J}^+ $, then $ \idl{J}^*\cup\{A\} $ has the finite intersection property, which generates a proper filter. We extend it to some $ \U\supseteq \idl{J}^*\cup\{A\}$. So this $ \U $ has $ \U \in \what{A} $.
		
		\item Straightforward.
		
		\item by (2), $ \what{X}\cap O_h(\idl{E})=\bigcup_{f\in\idl{E}}\left(\what{X}\cap\what{\Eq{D}(h,f)}\right)=\bigcup_{f\in\idl{E}}\what{X\cap\Eq{D}(h,f)} $.
		
		\item This is because
		\[
		\begin{aligned}
			\bigcap_{i<k}O_{h_i}(\idl{E})
			&=\bigcap_{i<k}\bigcup_{f\in\idl{E}}\what{\Eq{D}(h_i,f)}\\
			&=\bigcup_{\la f_0,\ldots,f_{k-1}\ra\in\idl{E}^k}\bigcap_{i<k}\what{\Eq{D}(h_i,f_i)}\\
			&=\bigcup_{\la f_0,\ldots,f_{k-1}\ra\in\idl{E}^k}\what{\bigcap_{i<k}\Eq{D}(h_i,f_i)}.
		\end{aligned}
		\]
		
		\item Suppose $ f\neq g\in\idl{E} $ and some $ \U\in \what{\Eq{D}(h,f)}  \cap   \what{\Eq{D}(h,g)} $. Then $ \Eq{D}(h,f)\cap\Eq{D}(h,g) \in \U $, while $ \Eq{D}(h,f)\cap\Eq{D}(h,g) \subseteq \Eq{D}(f,g)\in\idl{J} $, a contradiction.
	\end{enumerate}
\end{proof}

\begin{thm}\label{thm:stone_hierarchy}
	Let $ \idl{E}\subseteq\prodspc{F} $ be $ \idl{J} $-e.d.. Then the following hold.
	\begin{enumerate}[label=(\arabic*),font=\upshape]
		\item The family $ \idl{E} $ is $ \idl{J} $-MED if and only if $\forall h\in\prodspc{F}, \ O_h(\idl{E})\neq\emptyset $.

		\item The family $ \idl{E} $ is \FP $ \idl{J} $-MED if and only if $ \{O_h(\idl{E}):h\in\prodspc{F}\} $ has the finite intersection property.
		
		\item If $ \idl{E} $ is \omegaCP $ \idl{J} $-MED, then every countable subfamily of $ \{O_h(\idl{E}):h\in\prodspc{F}\} $ has non-empty intersection. If $ \idl{J} $ is weakly $ P^+ $, then the converse also holds.
		
		\item The family $ \idl{E} $ is \FSP $ \idl{J} $-MED if and only if $ \bigcap_{h\in\prodspc{F}}O_h(\idl{E})\neq\emptyset $.
		
		\item The family $ \idl{E} $ is \VD $ \idl{J} $-MED if and only if $ \forall  h\in\prodspc{F} $, $ O_h(\idl{E}) $ is dense in $ \beta_\idl{J}D $.
	\end{enumerate}
\end{thm}
\begin{proof}
	\phantom{a}
	
	\begin{enumerate}[label=(\arabic*),font=\upshape]
		\item The family $ \idl{E} $ is $ \idl{J} $-MED if and only if $\forall  h\in\prodspc{F}\ \exists f\in\idl{E} $ such that $ \Eq{D}(h,f)\in\idl{J}^+ $. By $ O_h(\idl{E})=\bigcup_{f\in\idl{E}}\what{\Eq{D}(h,f)} $
		and Fact \ref{fact:stone_basic}(1), it is non-empty if and only if $\exists f\in \idl{E} \  \Eq{D}(h,f)\in\idl{J}^+ $.
		
		\item Fix $ 1\leq k<\omega $, $ h_0,\ldots,h_{k-1}\in\prodspc{F} $ and check Fact \ref{fact:stone_basic}(4). By Fact \ref{fact:stone_basic}(1), the right hand side is non-empty if and only if there are $ f_0,\ldots,f_{k-1}\in\idl{E} $ such that $ \bigcap_{i<k}\Eq{D}(h_i,f_i)\in\idl{J}^+ $. Thus every finite intersection of members of $ \{O_h(\idl{E}):h\in\prodspc{F}\} $ is non-empty if and only if $ \idl{E} $ is \FP $ \idl{J} $-MED.
		
		\item \begin{itemize}
				\item[i.] Suppose $ \idl{E} $ is \omegaCP $ \idl{J} $-MED, and let $ \la h_i:i\in\omega\ra $ be a sequence in $ \prodspc{F} $. There are $ B\in\idl{J}^+ $ and $ f_i\in\idl{E} $ such that $ \forall i\in\omega \  B\subseteq_{\idl{J}}\Eq{D}(h_i,f_i) $. Choose a $ \U\in\what{B} $. For every $ i\in\omega $, we have
				\[
				B^\compl \cup \Eq{D}(h_i,f_i) = ( B\setminus\Eq{D}(h_i,f_i) )^\compl\in \U.
				\]
				Since $ B^\compl\in\U^* $, we have $ \Eq{D}(h_i,f_i)\in\U $, i.e. the $ f_i $ witnesses $ \U\in O_{h_i}(\idl{E}) $. So for this $ \U $ we have $ 
				\U\in\bigcap_{i\in\omega}O_{h_i}(\idl{E}) $, i.e. every countable subfamily has non-empty intersection.
			
				\item[ii.] Conversely, suppose $ \idl{J} $ is weakly $ P^+ $ and every countable subfamily of $ \{O_h(\idl{E}):h\in\prodspc{F}\} $ has non-empty intersection. Let $ \la h_i:i\in\omega\ra $ be a sequence in $ \prodspc{F} $, and choose $ \U\in\bigcap_{i\in\omega}O_{h_i}(\idl{E}) $. For every $ i\in\omega $, there is $ f_i\in\idl{E} $ such that $ 
				\Eq{D}(h_i,f_i)\in\U $. For any $ n\in \omega $ let $ A_n=\bigcap_{i\leq n} \Eq{D}(h_i,f_i) \in \U$, which is decreasing and also in $ \idl{J}^+ $. So by weak $ P^+ $ property there is $ B\in\idl{J}^+ $ such that $ \forall n\in\omega \ B\subseteq_{\idl{J}}A_n \subseteq \Eq{D}(h_n,f_n) $. Thus $ \idl{E} $ is \omegaCP $ \idl{J} $-MED.
		\end{itemize}
		
		\item By Lemma \ref{lem:fsp_global_ultrafilter}, the $ \idl{J} $-e.d. family $ \idl{E} $ is \FSP $ \idl{J} $-MED if and only if there is an ultrafilter $ \U\in\beta_\idl{J}D $ such that $ \idl{E} $ is $ \U^* $-MED, which is equivalent to being $ \U^* $-maximal. And $ \U^* $-maximality is equivalent to $ \forall h\in\prodspc{F}\ \exists f\in\idl{E},\  \Eq{D}(h,f)\in\U $, i.e. $ \U\in\bigcap_{h\in\prodspc{F}}O_h(\idl{E}) $. Thus $ \idl{E} $ is \FSP $ \idl{J} $-MED if and only if the total intersection is non-empty.
		
		\item \begin{itemize}
			\item[i.] Suppose $ \idl{E} $ is \VD $ \idl{J} $-MED. Fix $ h\in\prodspc{F} $ and a non-empty basic clopen set $ \what{X} $ in $ \beta_\idl{J}D $. We have $ X\in\idl{J}^+ $. Applying \VD maximality to $ h\uphar X $, choose $ f\in\idl{E} $ such that $ \Eq{X}(h,f)\in\idl{J}^+ $. So $ \what{X\cap\Eq{D}(h,f)}\neq\emptyset $, and by Fact \ref{fact:stone_basic}(3),
			\[
			\emptyset \neq \what{X\cap\Eq{D}(h,f)}\subseteq\what{X}\cap O_h(\idl{E}).
			\]
			Thus every non-empty basic clopen set meets $ O_h(\idl{E}) $, i.e. $ O_h(\idl{E}) $ is dense in $ \beta_\idl{J}D $.
			
			\item[ii.] Conversely, suppose every $ O_h(\idl{E}) $ is dense. Let $ X\in\idl{J}^+ $ and $ p\in\prodspc{F\uphar X} $. Extend $ p $ to some $ h_p\in\prodspc{F} $. As the clopen set $ \what{X} $ is non-empty, by density we have $ \what{X}\cap O_{h_p}(\idl{E})\neq\emptyset $. By Fact \ref{fact:stone_basic}(3), $\exists f\in\idl{E} $ such that $ X\cap\Eq{D}(h_p,f)\in\idl{J}^+ $. This means $ \Eq{X}(p,f)=X\cap\Eq{D}(h_p,f)\in\idl{J}^+ $, and then $ \idl{E} $ is \VD $ \idl{J} $-MED.
		\end{itemize}
	\end{enumerate}
\end{proof}

	Theorem \ref{thm:stone_hierarchy} gives the following implication.

\begin{cor}\label{cor:productive_hierarchy}
	Let $ \idl{E}\subseteq\prodspc{F} $ be $ \idl{J} $-e.d.. Then we have the following implication.
	\[
	\begin{tikzpicture}[>=stealth,baseline=(current bounding box.center)]
		\node (fp)  at (0,0) {$\FP \idl{J}\text{-MED}$};
		\node (med) at (3.2,0) {$\idl{J}\text{-MED}$};
		\node (vd)  at (-3.5,0.75) {$\VD \idl{J}\text{-MED}$};
		\node (fsp) at (-3.5,0) {$\FSP \idl{J}\text{-MED}$};
		\node (ocp) at (-3.5,-0.75) {$\omegaCP \idl{J}\text{-MED}$};
		
		\draw[->] (vd) -- (fp);
		\draw[->] (fsp) -- (fp);
		\draw[->] (ocp) -- (fp);
		\draw[->] (fp) -- (med);
	\end{tikzpicture}
	\]
	
	If $ \idl{J} $ is weakly $ P^+ $, then we have the following implication.
	\[
	\begin{tikzpicture}[>=stealth,baseline=(current bounding box.center)]
		\node (ocp) at (0,0) {$\omegaCP\ \idl{J}\text{-MED}$};
		\node (fp)  at (3,0) {$\FP\ \idl{J}\text{-MED}$};
		\node (med) at (6,0) {$\idl{J}\text{-MED}$};
		\node (vd)  at (-4,0.5) {$\VD\ \idl{J}\text{-MED}$};
		\node (fsp) at (-4,-0.5) {$\FSP\ \idl{J}\text{-MED}$};
		
		\draw[->] (vd) -- (ocp);
		\draw[->] (fsp) -- (ocp);
		\draw[->] (ocp) -- (fp);
		\draw[->] (fp) -- (med);
	\end{tikzpicture}
	\]
\end{cor}
\begin{proof}
	The implications from \FSP and \omegaCP to \FP, and from \FP to ordinary maximality, follow from Proposition \ref{prop:elementary_productive_implications}.
	
	Suppose that $ \idl{E} $ is \VD $ \idl{J} $-MED. By Theorem \ref{thm:stone_hierarchy}(5), every $ O_h(\idl{E}) $ is open dense. Every finite intersection of open dense subsets of $ \beta_\idl{J}D $ is non-empty, so Theorem \ref{thm:stone_hierarchy}(2) implies that $ \idl{E} $ is \FP $ \idl{J} $-MED.
	
	Now suppose $ \idl{J} $ is weakly $ P^+ $. The space $ \beta_\idl{J}D $ is compact Hausdorff and is a Baire space. If $ \idl{E} $ is \VD, then every countable intersection of the open dense sets $ O_h(\idl{E}) $ is non-empty, so Theorem \ref{thm:stone_hierarchy}(3) implies that $ \idl{E} $ is \omegaCP. If $ \idl{E} $ is \FSP, then Theorem \ref{thm:stone_hierarchy}(4) gives a point in the total intersection $ \bigcap_{h\in\prodspc{F}}O_h(\idl{E}) $, so every countable intersection is non-empty, and Theorem \ref{thm:stone_hierarchy}(3) again implies that $ \idl{E} $ is \omegaCP.
\end{proof}

The cardinal invariant $ \mdl{p}_{\idl{J}} $ controls how long a $ \idl{J}^+ $-FIP family can be fused by positive pseudointersections. We use this to obtain a sufficient condition under which \VD maximality implies the global ultrafilter condition.

\begin{thm}\label{thm:pj_vd_fsp}
	Assume that $ |\prodspc{F}|\leq\mdl{p}_{\idl{J}} $. Then every \VD $ \idl{J} $-MED family in $ \prodspc{F} $ is \FSP $ \idl{J} $-MED.
\end{thm}
\begin{proof}
	Let $ \kappa=|\prodspc{F}| $, and enumerate $ \prodspc{F}=\{h_\alpha:\alpha<\kappa\} $. We recursively define $ C_\alpha,B_\alpha\in\idl{J}^+ $ and $ f_\alpha\in\idl{E} $ for every $ \alpha<\kappa $ such that
	\begin{align}
		B_\alpha= \ &C_\alpha\cap\Eq{D}(h_\alpha,f_\alpha),\label{eq:pj_vd_exact_catch}\\
		\forall\beta<\alpha,\ &C_\alpha\subseteq_{\idl{J}}B_\beta,\label{eq:pj_vd_pseudointersection}\\
		\forall\beta<\alpha,\ &B_\alpha\subseteq_{\idl{J}}B_\beta.\label{eq:pj_vd_decreasing}
	\end{align}
	
	First let $ C_0=D $. By \VD maximality applied to $ h_0\uphar C_0=h_0 $, choose $ f_0\in\idl{E} $ such that
	\[
	B_0:=C_0\cap\Eq{D}(h_0,f_0)=\Eq{D}(h_0,f_0)\in\idl{J}^+.
	\]
	Thus the requirements hold at $\alpha =  0 $.
	
	Now let $ 0<\alpha<\kappa $, and suppose $ C_\beta,B_\beta\in\idl{J}^+ $ and $ f_\beta\in\idl{E} $ have been defined for every $ \beta<\alpha $ such that \eqref{eq:pj_vd_exact_catch}--\eqref{eq:pj_vd_decreasing} hold also for any $ \beta<\alpha $.
	
	We first show that $ \{B_\beta:\beta<\alpha\} $ has $ \idl{J}^+ $-FIP. Let $ \beta_0<\cdots<\beta_{k-1}<\alpha $. By \eqref{eq:pj_vd_decreasing}, for every $ i<k-1 $, $ B_{\beta_{k-1}}\setminus B_{\beta_i}\in\idl{J} $. So
	\[
	B_{\beta_{k-1}}\setminus\bigcap_{i<k}B_{\beta_i}  = \bigcup_{i<k-1}(B_{\beta_{k-1}}\setminus B_{\beta_i})\in\idl{J}.
	\]
	Since $ B_{\beta_{k-1}}\in\idl{J}^+ $, we have that $ \bigcap_{i<k}B_{\beta_i}\in\idl{J}^+ $, and then $ \{B_\beta:\beta<\alpha\} $ has $ \idl{J}^+ $-FIP.
	
	Also, $|\{B_\beta:\beta<\alpha\}|\leq|\alpha|<\kappa\leq\mdl{p}_{\idl{J}} $, so by the definition of $ \mdl{p}_{\idl{J}} $, the family $ \{B_\beta:\beta<\alpha\} $ has a $ \idl{J}^+ $-pseudointersection $ C_\alpha\in\idl{J}^+ $. Thus $ \forall\beta<\alpha,\ C_\alpha\subseteq_{\idl{J}}B_\beta$,  this way \eqref{eq:pj_vd_pseudointersection} holds.
	
	Apply \VD maximality to $ h_\alpha\uphar C_\alpha $. Choose $ f_\alpha\in\idl{E} $ such that
	\[
	B_\alpha:=C_\alpha\cap\Eq{D}(h_\alpha,f_\alpha)\in\idl{J}^+.
	\]
	This gives \eqref{eq:pj_vd_exact_catch}. This way for every $ \beta<\alpha $ we have $ B_\alpha\subseteq_{\idl{J}}B_\beta $, so \eqref{eq:pj_vd_decreasing} also holds. This completes the recursion.
	
	We claim that $ \idl{J}^*\cup\{B_\alpha:\alpha<\kappa\} $ has the finite intersection property. Let $ \alpha_0<\cdots<\alpha_{k-1}<\kappa $, and $ J_0,\ldots,J_{l-1}\in\idl{J} $. By the same argument above, $ \bigcap_{i<k}B_{\alpha_i}\in\idl{J}^+ $, so
	\[
	\left(\bigcap_{i<k}B_{\alpha_i}\right)\cap\bigcap_{j<l}(D\setminus J_j)
	=
	\left(\bigcap_{i<k}B_{\alpha_i}\right)\setminus\bigcup_{j<l}J_j
	\in\idl{J}^+,
	\]
	which means $ \idl{J}^*\cup\{B_\alpha:\alpha<\kappa\} $ has finite intersection property. Extend the proper filter generated by this family to an ultrafilter $ \U $ on $ D $. Then $ \idl{J}^*\subseteq\U $ and $\forall \alpha <\kappa, \  B_\alpha\in\U $. By \eqref{eq:pj_vd_exact_catch}, we have $ \forall \alpha<\kappa,\ \Eq{D}(h_\alpha,f_\alpha)\in\U  $. Since $ \idl{J}\subseteq\U^* $, the family $ \idl{E} $ is also $ \U^* $-e.d.. We found a $ \U $ such that the family $ \idl{E} $ is $ \U^* $-MED. By Lemma \ref{lem:fsp_global_ultrafilter}, $ \idl{E} $ is \FSP $ \idl{J} $-MED.
\end{proof}

Especially, let $ \chaxiom $ denote the \emph{Continuum Hypothesis}, and $ \mathfrak{p} $ be the \emph{pseudointersection number}. we have the following.

\begin{cor}\label{cor:ch_vd_fsp}
	\phantom{a}
	
	\begin{itemize}[leftmargin=1.5cm, labelwidth=1cm, labelsep=0.2cm, align=right]
		\item[\textnormal{($ \chaxiom $)}] Suppose that $ \idl{J} $ is weakly $ P^+ $. Then every \VD $ \idl{J} $-MED family in $ \prodspc{F} $ is \FSP $ \idl{J} $-MED.
		\item[\textnormal{($ \mathfrak{p}=\mathfrak{c} $)}] For any $ F_\sigma $ ideal $ \idl{J} $, every \VD $ \idl{J} $-MED family in $ \prodspc{F} $ is \FSP $ \idl{J} $-MED.
	\end{itemize}

\end{cor}
\begin{proof}
	\phantom{a}
	
	\begin{itemize}
		\item[(1)] By Proposition \ref{prop:weak_p_plus_pj}, we have $ \mdl{p}_{\idl{J}}>\omega $. If $ \mdl{p}_{\idl{J}}=\mathfrak{c}^+ $, then the hypothesis of Theorem \ref{thm:pj_vd_fsp} is automatic. Otherwise, $ \mdl{p}_{\idl{J}} $ is an uncountable cardinal, so $ \mdl{p}_{\idl{J}}\geq\omega_1=\mathfrak{c} $ by $ \chaxiom $. Since $ D $ is countable, we also have $ |\prodspc{F}|\leq|\omega^D|\leq\mathfrak{c}$. Thus $ |\prodspc{F}|\leq\mdl{p}_{\idl{J}} $, and Theorem \ref{thm:pj_vd_fsp} applies.
		\item[(2)] We know that for any $ F_\sigma $ ideal $ \idl{J} $, $ \mathfrak{p}\leq \mathfrak{p}_\idl{J} $ \cite{filipow2022characterizing}. So $ \mathfrak{p}=\mathfrak{c}  $ implies $  |\prodspc{F}|\leq\mathfrak{p}_\idl{J} $, and Theorem \ref{thm:pj_vd_fsp} applies.
	\end{itemize}
\end{proof}

	\subsection{Fubini powers of Ramsey ultrafilters}

We first note that, for the dual ideal of an ultrafilter, ordinary idealized maximality and Van Douwen maximality coincide.

\begin{fact}\label{fact:ultrafilter_med_vd}
	Let $ F $ be a partial function, let $ D=\dom(F) $, and let $ \U $ be an ultrafilter on $ D $. A $ \U^* $-e.d. family $ \idl{E}\subseteq\prodspc{F} $ is $ \U^* $-MED if and only if it is \VD $ \U^* $-MED.
\end{fact}
\begin{proof}
	Suppose $ \idl{E} $ is $ \U^* $-MED. Let $ X\in\U $ and let $ h\in\prodspc{F\uphar X} $. Extend $ h $ to some $ \widehat{h}\in\prodspc{F} $, and choose $ f\in\idl{E} $ such that $ \Eq{D}\left(\widehat{h},f\right)\in\U $. Then $ \Eq{X}(h,f)=X\cap\Eq{D}\left( \widehat{h},f\right)\in\U $. Thus $ \idl{E} $ is \VD $ \U^* $-MED. The converse follows by taking $ X=D $.
\end{proof}

\begin{defn}\label{def:tilde_fubini_basis}
	For $ H\in\infsub{\omega} $ and $ \xi<\omega_1 $, define $ \wbar{H}^{\xi}\subseteq S_\xi $ recursively by $ \wbar{H}^{0}=\{\emptyset\} $ and for $ \xi\geq 1 $,
	\[
	\wbar{H}^{\xi}=\{n^\frown s:n\in H\text{ and }s\in\wbar{H\setminus(n+1)}^{\gamma_n^\xi}\},
	\]
	where $ H\setminus(n+1)=\{m\in H:m>n\} $.
\end{defn}

By the definition, we have that $ \dom(\wbar{H}^\xi ) \subseteq H $ and $ \forall n \in H $, $ \wbar{H}^\xi (n) = \wbar{H\setminus(n+1)}^{\gamma_n^\xi} $.

\begin{fact}\label{fact:tilde_elementary}
	Let $ \xi<\omega_1 $ and $ H,K\in\infsub{\omega} $.
	\begin{enumerate}[label=(\arabic*),font=\upshape]
		\item A sequence $ s\in S_\xi $ belongs to $ \wbar{H}^{\xi} $ if and only if $ s $ is strictly increasing and $ \ran(s)\subseteq H $. So for any $ n\in H $ we have $ \wbar{H\setminus(n+1)}^{\xi}\neq\emptyset $.
		\item If $ \xi\geq 1 $, then $ \dom(\wbar{H}^{\xi})=H $.
		\item If $ K\subseteq H $, then $ \wbar{K}^{\xi}\subseteq\wbar{H}^{\xi} $.
		\item The map $ H\mapsto\wbar{H}^{\xi} $ from $ \infsub{\omega} $ to $ \prst(S_\xi) $ is continuous.
	\end{enumerate}
\end{fact}
\begin{proof}
	We prove the assertions by induction on $ \xi $. The case $ \xi=0 $ for (1), (3), (4) and the case $ \xi = 1 $ for (2) are immediate.
	\begin{enumerate}[label=(\arabic*),font=\upshape]
		\item Suppose $ \xi\geq1 $ and write $ s=n^\frown t\in S_\xi $. By Definition \ref{def:tilde_fubini_basis}, $ s\in\wbar{H}^{\xi} $ if and only if  $ n\in H\ \wedge\ t\in\wbar{H\setminus(n+1)}^{\gamma_n^\xi} $. By the induction hypothesis, this is equivalent to $ n\in H \ \wedge \ \ran(t) \subseteq H\setminus (n+1) \ \wedge  \ t$ is strictly increasing, which is further equivalent to $ s $ is strictly increasing $\ \wedge \ \ran(s)\subseteq H $.
		
		\item By Fact \ref{fact:tilde_elementary}(1).
		\item By Fact \ref{fact:tilde_elementary}(1).
		\item The membership of a fixed $ s\in S_\xi $ in $ \wbar{H}^{\xi} $ depends only on the finitely many statements $ n\in H $ for $ n\in\ran(s) $.
	\end{enumerate}
\end{proof}

\begin{prop}\label{prop:tilde_basis}
	Let $ \U $ be a Ramsey ultrafilter and let $ \xi<\omega_1 $.
	\begin{enumerate}[label=(\arabic*),font=\upshape]
		\item If $ H\in\U $, then $ \wbar{H}^{\xi}\in\U^\xi $.
		\item If $ X\in\U^\xi $, then there is $ H\in\U $ such that $ \wbar{H}^{\xi}\subseteq X $.
	\end{enumerate}
\end{prop}
\begin{proof}
	We first relativize the diagonalization property which will be used in the induction.
	
	\begin{claim}\label{claim:ramsey_diagonalization}
		If $ M\in\U $ and $ H_n\in\U $ for every $ n\in\omega $, then there is $ H\in\U $ such that $ H\subseteq M $ and $ H\setminus(n+1)\subseteq H_n $ for every $ n\in H $.
	\end{claim}
	\begin{claimproof}{\ref{claim:ramsey_diagonalization}}
		Define $ c:[\omega]^2\to2 $ by $ c(\{n,m\})=0 $ for $ n<m $ if and only if $ m\in H_n $. Get a homogeneous $ H_0\in\U $, and let $ H=H_0\cap M $. Claim that $ H $ is  0-homogeneous. This is because for any $ n\in H $, as $ H\cap H_n\in\U $, there is $ m\in H\cap H_n $ with $ m>n $ i.e. $ c(\{n,m\})=0 $, so $ H $ can not be 1-homogeneous and thus is 0-homogeneous. This means for any $ n\in H $, $ H\setminus(n+1)\subseteq H_n $.
	\end{claimproof}
	
	We prove by induction on $ \xi $ simultaneously for (1) and (2). The case $ \xi=0 $ is trivial. Suppose $ \xi\geq 1 $ and both assertions hold below $ \xi $. 
	\begin{enumerate}[label=(\arabic*),font=\upshape]
		\item For any $ H\in\U $, any every $ n\in H $, we have $ H\setminus(n+1)\in\U $, so by the induction hypothesis we have $ \wbar{H\setminus(n+1)}^{\gamma_n^\xi}\in\U^{\gamma_n^\xi}$. So $ \U\ni H\subseteq \{n\in\omega:\wbar{H}^{\xi}(n)\in\U^{\gamma_n^\xi}\}$, which means $ \wbar{H}^{\xi}\in\U^\xi $.
		
		\item Let $ X\in\U^\xi $ and let $ M=\{n\in\omega:X(n)\in\U^{\gamma_n^\xi}\}\in\U $. For every $ n\in M $, by induction hypothesis get an $ H_n\in\U $ such that $ \wbar{H_n}^{\gamma_n^\xi}\subseteq X(n) $. For $ n\notin M $, let $ H_n=\omega $. By Claim \ref{claim:ramsey_diagonalization}, there is $ H\in\U $ such that $ H\subseteq M $ and $\forall n\in H, \  H\setminus(n+1)\subseteq H_n $. So if $ n\in H $, then Fact \ref{fact:tilde_elementary} gives
		\[
		\wbar{H}^{\xi}(n)=\wbar{H\setminus(n+1)}^{\gamma_n^\xi}\subseteq\wbar{H_n}^{\gamma_n^\xi}\subseteq X(n).
		\]
		Also, if $ n\notin H $, as $ \wbar{H}^{\xi}(n) =\emptyset$, we have that $ \wbar{H}^{\xi}\subseteq X $.
	\end{enumerate}
\end{proof}

\begin{thm}\label{thm:no_analytic_ultrafilter_fubini_mad}
	Let $ \U $ be a Ramsey ultrafilter on $ \omega $, and let $ 1\leq\alpha<\omega_1 $. Then there is no infinite analytic $ (\U^\alpha)^*\times\emptyset $-MAD family on $ S_\alpha\times\omega $.
\end{thm}
\begin{proof}
	Suppose $ \idl{A}\subseteq\prst(S_\alpha\times\omega) $ is an infinite analytic $ (\U^\alpha)^*\times\emptyset $-MAD family. Choose pairwise distinct $ A_0,A_1,\cdots\in\idl{A} $. For every $ n\in\omega $, define $ g_n:S_\alpha\to\omega $ by
	\[
	g_n(s)=
	\begin{cases}
		\min A_n(s),&s\in\dom(A_n),\\
		0,&s\notin\dom(A_n).
	\end{cases}
	\]
	
	For any $ x\in\ramseysp $, define
	\[
	p_i^x=x(2i),\  q_i^x=x(2i+1),\ Q_x=\{q_i^x:i\in\omega\}\subseteq x.
	\]
	Also, for $ q_i^x\in Q_x $, let $ \ell_x(q_i^x)=p_i^x $, and define
	\begin{align*}
		\Phi(x):=&\{(s,g_{\ell_x(s(0))}(s)):s\in\wbar{Q_x}^{\alpha}\}\\
		=&\Bigg\{\Bigg( x'' \left(2\vec{i}+1 \right), g_{x(2i_0)}\bigg(x''\left(2\vec{i}+1 \right) \bigg) \Bigg):\\
		&\phantom{\Bigg\{\Bigg(}x''\left(2\vec{i}+1\right) \in S_\alpha \ \wedge \vec{i}\in\finsub{\omega}\Bigg\},\\
		\text{where } x''\left(2\vec{i}+1\right) =& \ x(2i_0+1)^\frown x(2i_1+1) ^\frown \cdots ^\frown x(2i_{|\vec{i}|-1}+1).
	\end{align*}

	By definition, the map $ \Phi:\ramseysp\to\prst(S_\alpha\times\omega) $ is continuous, and $ \dom(\Phi(x))=\wbar{Q_x}^{\alpha} $.
	 
	Define
	\[
	P=\{x\in\ramseysp:\exists A\in\idl{A}\ \Phi(x)\subseteq A\}.
	\]
	The set $ P $ is analytic.
	
	\begin{claim}\label{claim:ultrafilter_hiding}
		For every $ W\in \U $, there are $ y\in[W]^\omega $ and $ A\in\idl{A} $ such that $ Q_y\in\U $ and $ \Phi(y)\subseteq A $.
	\end{claim}
	\begin{claimproof}{\ref{claim:ultrafilter_hiding}}
		Enumerate $ W=\{w_0<w_1<\cdots\} $. Exactly one of $ \{w_{2i}:i\in\omega\} $ and $ \{w_{2i+1}:i\in\omega\} $ belongs to $ \U $. If the latter belongs to $ \U $, let $ x=W $; otherwise let $ x=W\setminus\{w_0\} $. Then $ Q_x\in\U $, so $ \dom(\Phi(x))=\wbar{Q_x}^{\alpha}\in\U^\alpha $ by Proposition \ref{prop:tilde_basis}.
		
		By maximality, choose $ A\in\idl{A} $ such that
		\[
		X:= \{s\in \wbar{Q_x}^{\alpha}: \Phi(x)(s)\in A(s)\} = \dom(\Phi(x)\cap A)\in\U^\alpha.
		\]
		By Proposition \ref{prop:tilde_basis}, choose $ H\in\U $ such that $ \wbar{H}^{\alpha}\subseteq X $. Since $ X\subseteq\wbar{Q_x}^{\alpha} $, Fact \ref{fact:tilde_elementary}(2) gives $ H\subseteq Q_x $.
		
		Suppose $ H = \{q^x_{i_0}< q^x_{i_1}<\cdots< q^x_{i_n}<\cdots\} $. Define
		\[
		y = \{p^x_{i_0} < q^x_{i_0}< p^x_{i_1}< q^x_{i_1}<\cdots< p^x_{i_n}< q^x_{i_n}<\cdots\} \in\infsub{x}\subseteq \infsub{W}.
		\]
		Then $ Q_y=H\in\U $, and for any $ s\in \wbar{Q_y}^\alpha $, we have that $ \ell_y(s(0)) = \ell_x(s(0)) $. As we also have $ \wbar{Q_y}^\alpha \subseteq X \subseteq \wbar{Q_x}^\alpha  $, we have that
		\[
		\forall s\in \wbar{Q_y}^\alpha , \ \Phi(y)(s) = g_{\ell_y(s(0))}(s) = g_{\ell_x(s(0))}(s) = \Phi(x)(s)\in A(s),
		\]
		so $ \Phi(y)\subseteq A $.
	\end{claimproof}
\begin{claim}\label{claim:ultrafilter_escape}
	For every $ W\in\U $, there are $ z\in[W]^\omega $ and $ A\in\idl{A} $ such that $ \dom(\Phi(z) \cap A)\in\U^\alpha $ and $ \Phi(z)\setminus A\neq\emptyset $.
\end{claim}
\begin{claimproof}{\ref{claim:ultrafilter_escape}}
	By Claim \ref{claim:ultrafilter_hiding}, choose
	\[
	x=\{p_0<q_0<p_1<q_1<\cdots\}\in[W]^\omega
	\]
	and $ A\in\idl{A} $ such that $ Q:=\{q_i:i\in\omega\}\in\U $ and $ \Phi(x)\subseteq A $.
	
	We choose an $ n\in W $ such that $ A_n \neq A $, so $ \dom(A\cap A_n)\in(\U^\alpha)^* $. Since $ \dom(A_n)\in\U^\alpha $ and $ \U^\alpha $ is an ultrafilter by Fact \ref{fact:fubini_power_ultrafilter}, we have $D_n:=\dom(A_n)\setminus\dom(A\cap A_n)\in \U^\alpha $. By Proposition \ref{prop:tilde_basis}, get $ H\in\U $ such that $ \wbar{H}^{\alpha}\subseteq D_n $. Choose $ r\in W\cap H $ with $ r>n $, and let $ Q' = \{q_i: q_i\in H \ \wedge p_i>r\} $. As $ Q\cap H \in\U$ and there are only finitely many $ q_i $'s with $ p_i\leq r $, we have that $ Q'\in\U $. Say $ Q' = \{ q_{i_0}<q_{i_1}<\cdots q_{i_n}<\cdots\} $.
	
	Define
	\[
	z=\{n<r<p_{i_0}<q_{i_0}<p_{i_1}<q_{i_1}<\cdots\}.
	\]
	Then $ z\in[W]^\omega $ and $ Q_z=\{r\}\cup Q'\subseteq H $.
	
	For every $ s\in\wbar{Q'}^{\alpha} $, again we have that $ \ell_z(s(0)) = \ell_x(s(0))  $, and that $ \Phi(z)(s) = \Phi(x)(s)\in A(s) $, as $ \Phi(x)\subseteq A $. We have $ \wbar{Q'}^{\alpha}\subseteq\dom(A\cap\Phi(z)) $. As $ Q'\in\U $, Proposition \ref{prop:tilde_basis} gives $ \wbar{Q'}^{\alpha}\in\U^\alpha $, which implies
	\[
	\dom(A\cap\Phi(z))\in\U^\alpha.
	\]
	
	For the other side, since $ r\in Q_z $, by Fact \ref{fact:tilde_elementary}(2) there is $ s\in\wbar{Q_z}^{\alpha} $ with $ s(0)=r $. Note that $ s\in\wbar{Q_z}^{\alpha}\subseteq\wbar{H}^{\alpha}\subseteq D_n $. This means $ s\in \dom(A_n)$ while $ A(s)\cap A_n(s)=\emptyset $. Also, we have that $ \Phi(z)(s)= g_n(s) $. By the definition of $ g_n $, $ \Phi(z)(s) = \min A_n(s) \notin A(s)$. This means $ \Phi(z)\setminus A\neq\emptyset $.
\end{claimproof}

By Theorem \ref{thm:ramsey_ultrafilter}, there is $ W\in\U $ such that either $ [W]^\omega\subseteq P $ or $ [W]^\omega\cap P=\emptyset $.
\begin{itemize}
	\item[(1)] Suppose $ [W]^\omega\cap P=\emptyset $. By Claim \ref{claim:ultrafilter_hiding}, there are $ y\in[W]^\omega $ and $ A\in\idl{A} $ such that $ \Phi(y)\subseteq A $. Thus $ y\in P $, a contradiction.
	
	\item[(2)] Suppose $ [W]^\omega\subseteq P $. By Claim \ref{claim:ultrafilter_escape}, there are $ z\in[W]^\omega $ and $ A\in\idl{A} $ such that $ \dom(\Phi(z) \cap A)\in\U^\alpha $ and $ \Phi(z)\setminus A\neq\emptyset  $. Since $ z\in P $, there is $ C\in\idl{A} $ such that $ \Phi(z)\subseteq C $. Then $\dom(A\cap C) \supseteq\dom(A\cap\Phi(z)) \in \U^\alpha$, which implies $ A=C $. Hence $ \Phi(z)\subseteq A $, contradicting $ \Phi(z)\setminus A\neq\emptyset $.
\end{itemize}
So both cases lead to a contradiction, which finishes the proof.
\end{proof}

\begin{cor}\label{cor:no_analytic_ramsey_ultrafilter_med}
	Let $ \U $ be a Ramsey ultrafilter on $ \omega $, and let $ 1\leq\alpha<\omega_1 $. Then there is no analytic $ (\U^\alpha)^* $-MED family in $ \omega^{S_\alpha} $.
\end{cor}
\begin{proof}
	Suppose that $ \idl{E}\subseteq\omega^{S_\alpha} $ is analytic and $ (\U^\alpha)^* $-MED. The family $ \idl{E} $ is infinite, since any finite family can be avoided pointwise. By Fact \ref{fact:ultrafilter_med_vd}, the family $ \idl{E} $ is \VD $ (\U^\alpha)^* $-MED. By Proposition \ref{prop:vdj_jtimesempty}, the family $ \gr''\idl{E} $ is an infinite analytic $ (\U^\alpha)^*\times\emptyset $-MAD family on $ S_\alpha\times\omega $, contradicting Theorem \ref{thm:no_analytic_ultrafilter_fubini_mad}.
\end{proof}
	
	It's of interest to consider the existence of analytic $ \U^* $-MED families where $ \U $ is a P-point. Recall that an ideal $ \idl{J}\supseteq\mathrm{Fin} $ is a
	$ P $-ideal if, for every sequence
	$ \la A_n:n\in\omega\ra $ in $ \idl{J} $, there is some
	$ A\in\idl{J} $ such that $ \forall n\in\omega, \ A_n\subseteq^* A  $.
	
	\begin{prop}\label{prop:analytic_pideal_p_point_fsigma}
		Consider the following statements.
		\begin{enumerate}[label=(\arabic*),font=\upshape]
			\item There are an (analytic) $ P $-ideal $ \idl{J}\supseteq \mathrm{Fin} $ on
			$ \omega $ and an analytic \FSP $ \idl{J} $-MED family in
			$ \bairesp $.
			
			\item There are a $ P $-point $ \U $ on $ \omega $ and an
			analytic $ \U^* $-MED family in $ \bairesp $.
			
			\item There are an $ F_\sigma $ ideal $ \idl{I} $ on
			$ \omega $ and an analytic \FSP $ \idl{I} $-MED family in
			$ \bairesp $.
		\end{enumerate}
		Then (1) implies (2) and (2) implies (3).
	\end{prop}
	
	\begin{proof}
		\phantom{a}
		
		\begin{itemize}
			\item[(1)] Assume (1). Similar to Theorem \ref{thm:no_analytic_fsp_fin} and Claim \ref{claim:u_p_point}, we get a corresponding ultrafilter $ \U $ containing $ \idl{J}^* $ such that $ \idl{E}\subseteq\bairesp $ is $ \U^* $-MED, and that $ \U $ is a P-point: in our case, we fix $ P_n $, choose $ f_n $, define $ h $, choose $ f $ and define $ B $ exactly the same way, and we now have $ \forall n\in\omega, B\cap P_n \in \idl{J}  $. As $ \idl{J} $ is a $ P $-ideal, choose $ A\in\idl{J} $ such that $ \forall n\in\omega, \  B\cap P_n\subseteq^*A $. Now, let $ X = B\setminus A  = B\cap A^\compl \in \U$. Also, $ \forall n\in\omega , \ X\cap P_n = B\cap P_n \setminus A $, which is finite. So $ \U  $ is a P-point witnessing (2).
			
			\item[(2)] Assume (2). Let $ \U $ be a $ P $-point and let $ \idl{E}\subseteq\bairesp $ be an analytic $ \U^* $-MED family. Let $ \idl{K} $ be the ideal generated by
			\[
			\mathrm{Fin}\cup \left\{ \Eq{\omega}(f,g): f,g\in\idl{E}\text{ and }f\neq g \right\}.
			\]
			So $ \idl{K} $ is analytic.
			
			Since $ \idl{E} $ is $ \U^* $-e.d., we have $ \forall f\neq g \in\idl{E}, \ \Eq{\omega}(f,g)\in\U^* $; also $ \mathrm{Fin}\subseteq \U^* $. So $ \idl{K}\subseteq\U^* $ and further $ \idl{K}\cap \U =\emptyset$. By \cite[Claim~2.4]{zapletal2009preserving}, there is an $ F_\sigma $ ideal $ \idl{I} \supseteq \idl{K}$ such that $ \idl{I}\cap\U =\emptyset$, which is $ \idl{I}\subseteq \U^* $. Also, for any $ f\neq g\in\idl{E} $, we have $ \Eq{\omega}(f,g)\in\idl{K}\subseteq\idl{I} $, so $ \idl{E} $ is $ \idl{I} $-e.d.. By Lemma
			\ref{lem:fsp_global_ultrafilter}, we have that $ \idl{E} $ is an analytic \FSP $ \idl{I} $-MED family. So (2) implies (3).
		\end{itemize}
	\end{proof}

		\subsection{Coanalytic $ \U^* $-MED families in $ L $}
	
	We now show that the analytic bound cannot be improved in $ L $.
	
	\begin{defn}\label{def:vl_coding}
		Let $ M=\bairesp $, let $ E=\{2n:n\in\omega\} $, and fix a recursive injection
		\[
		C:2\times\ramseysp\times\bairesp\times\bairesp\to\bairesp
		\]
		such that $ \ran(C) $ is recursive, the coordinate maps $ \tagg,\base,\real,\degree $ are partial recursive on $ \ran(C) $, and whenever $ y=C(i,A,x,r) $, we have $ \tagg(y)=i $, $ \base(y)=A $, $ \real(y)=x $, $ \degree(y)=r $, and $ y(2n)=x(2n) $ for every $ n\in\omega $. Thus $ 	E\subseteq\Eq{\omega}(x,y) $ whenever $ y=C(i,A,x,r) $.
		
		Let $ \Breq $ be the disjoint union of four standard Borel spaces coding, respectively, reals $ x\in\bairesp $, sets $ Y\subseteq\omega $, partitions $ \omega=\bigcup_{n\in\omega}P_n $ into finite pieces, and colorings $ [\omega]^2\to2 $. For $ s\in M^{\leq\omega} $, define
		\[
		\rec(s)=\{z\in\ran(s):z\in\ran(C)\text{ and }\tagg(z)=1\}.
		\]
		We say that $ s $ is \textbf{good}, and write $ \good(s) $, if every $ z\in\ran(s) $ belongs to $ \ran(C) $, every $ \base(z)\subseteq^*E $, and the family $ \{\base(z):z\in\ran(s)\} $ has \textbf{strong finite intersection property} (SFIP) i.e.every collection of finitely many elements has infinite intesection.
	\end{defn}
	
	We use the following coanalytic transfinite construction theorem.
	
	\begin{thm}[{\cite{vidnyanszky2014transfinite}}]\label{thm:vidnyanszky_coanalytic}
		Assume $ V=L $. Let $ B $ be an uncountable Borel subset of a Polish space and let $ M $ be a Polish space. Suppose $ R\subseteq M^{\leq\omega}\times B\times M $ is coanalytic and every section
		\[
		R(s,p)=\{y\in M:(s,p,y)\in R\}
		\]
		is cofinal in the Turing degrees. Then there are enumerations $ B=\{p_\xi:\xi<\omega_1\} $ and $ X=\{y_\xi:\xi<\omega_1\} $, where $ X\subseteq M $ is coanalytic, such that for every $ \xi<\omega_1 $ there is an enumeration $ s_\xi $ of $ \{y_\eta:\eta<\xi\} $ satisfying $ y_\xi\in R(s_\xi,p_\xi) $.
	\end{thm}
	
	\begin{thm}\label{thm:vl_ramsey_coanalytic_med}
		Assume $ V=L $. There is a $ \bfsgm^1_2 $ Ramsey ultrafilter $ \U $ on $ \omega $ and a coanalytic $ \U^* $-MED family $ \idl{E}\subseteq\bairesp $.
	\end{thm}
	\begin{proof}
		Define $ R\subseteq M^{\leq\omega}\times\Breq\times M $ as follows. If $ \neg\good(s) $, then $ (s,p,y)\in R $ if and only if $ y\in\ran(C) $, $ \tagg(y)=0 $, and $ \base(y)\subseteq^*E $. Suppose $ \good(s) $. Then $ (s,p,y)\in R $ if and only if $ y=C(i,A,x,r) $ and the following conditions hold.
		\begin{enumerate}[label=(M\arabic*),font=\upshape]
			\item $ A\subseteq^*E $, and $ A\subseteq^*\base(z) $ for every $ z\in\ran(s) $.
			\item If $ p $ codes a set $ Y\subseteq\omega $, then $ A\subseteq^*Y $ or $ A\subseteq^*(\omega\setminus Y) $. If $ p $ codes a partition $ \omega=\bigcup_{n\in\omega}P_n $ into finite pieces, then $ |A\cap P_n|\leq1 $ for every $ n\in\omega $. If $ p $ codes a coloring $ c:[\omega]^2\to2 $, then $ A $ is homogeneous for $ c $. If $ p $ codes a real $ x_p $, then $ x=x_p $.
			\item If $ p $ codes a real $ x_p $, then for every $ z\in\rec(s) $ exactly one of $ A\subseteq^*\Eq{\omega}(x_p,z) $ and $ A\subseteq^*(\omega\setminus\Eq{\omega}(x_p,z)) $ holds, and $ i=0$ if and only if $\exists z\in\rec(s)\ A\subseteq^*\Eq{\omega}(x_p,z) $.
			\item If $ p $ does not code a real, then $ i=0 $.
		\end{enumerate}
		
		\begin{claim}\label{claim:vl_relation}
			The relation $ R $ is Borel, and every section $ R(s,p) $ is cofinal in the Turing degrees.
		\end{claim}
		\begin{claimproof}{\ref{claim:vl_relation}}
			The coding maps are recursive. The relations $ A\subseteq^*B $, $ A\subseteq^*\Eq{\omega}(x,z) $, and $ A\subseteq^*(\omega\setminus\Eq{\omega}(x,z)) $ are Borel. Meeting every member of a partition into finite pieces in at most one point and being homogeneous for a fixed coloring are Borel conditions, and $ \good(s) $ is a Borel condition. Thus $ R $ is Borel.
			
			Fix $ s\in M^{\leq\omega} $, $ p\in\Breq $, and $ w\in\bairesp $. If $ \neg\good(s) $, let $ y=C(0,E,0,w) $, where 0 denotes the zero real. Then $ y\in R(s,p) $, and $ w\leq_Ty $ since $ \degree(y)=w $.
			
			Suppose $ \good(s) $. Since $ \{\base(z):z\in\ran(s)\}\cup\{E\} $ is a countable SFIP family of infinite subsets of $ \omega $, choose an infinite pseudointersection $ P\subseteq E $. We refine $ P $ according to the requirement coded by $ p $.
			\begin{itemize}
				\item[(1)] If $ p $ codes a set $ Y\subseteq\omega $, pass to an infinite subset of $ P $ contained in $ Y $ or in $ \omega\setminus Y $.
				\item[(2)] If $ p $ codes a partition $ \omega=\bigcup_{n\in\omega}P_n $ into finite pieces, pass to an infinite subset of $ P $ meeting every $ P_n $ in at most one point.
				\item[(3)] If $ p $ codes a coloring $ c:[\omega]^2\to2 $, pass to an infinite homogeneous subset of $ P $.
				\item[(4)] If $ p $ codes a real, leave $ P $ unchanged at this stage.
			\end{itemize}
			
			Call the resulting infinite set $ P' $. If $ p $ codes a real $ x_p $, enumerate $ \rec(s)=\{z_n:n\in I\} $, where $ I $ is an initial segment of $ \omega $ or all of $ \omega $. Starting with $ P_0=P' $, recursively choose an infinite $ P_{n+1}\subseteq P_n $ such that
			\[
			P_{n+1}\subseteq\Eq{\omega}(x_p,z_n)
			\]
			or
			\[
			P_{n+1}\subseteq\omega\setminus\Eq{\omega}(x_p,z_n).
			\]
			Choose an infinite pseudointersection $ A $ of this decreasing sequence. If $ p $ does not code a real, let $ A=P' $.
			
			Choose $ i $ according to (M3) or (M4), let $ x=x_p $ if $ p $ codes a real and let $ x=0 $ otherwise, and define $ y=C(i,A,x,w) $. Then $ y\in R(s,p) $, and again $ w\leq_Ty $.
		\end{claimproof}
		
		Apply Theorem \ref{thm:vidnyanszky_coanalytic} to $ R $. Let
		\[
		\Breq=\{p_\xi:\xi<\omega_1\},\qquad X=\{y_\xi:\xi<\omega_1\}
		\]
		be as given by the theorem, where $ X $ is coanalytic, and let $ s_\xi $ enumerate $ \{y_\eta:\eta<\xi\} $ such that $ y_\xi\in R(s_\xi,p_\xi) $. Define
		\[
		\U=\{Y\subseteq\omega:\exists\xi<\omega_1\ \base(y_\xi)\subseteq^*Y\},
		\qquad
		\idl{E}=\{y\in X:\tagg(y)=1\}.
		\]
		
		\begin{claim}\label{claim:vl_coherence}
			For every $ \xi<\omega_1 $, we have $ \good(s_\xi) $, $ \base(y_\xi)\subseteq^*E $, and $ \base(y_\xi)\subseteq^*\base(y_\eta) $ for every $ \eta<\xi $.
		\end{claim}
		\begin{claimproof}{\ref{claim:vl_coherence}}
			We argue by induction on $ \xi $. Suppose that the conclusion holds below $ \xi $. Then the bases $ \base(y_\eta) $, $ \eta<\xi $, are $ \subseteq^* $-decreasing and are all almost contained in $ E $. Thus they form a SFIP family and $ \good(s_\xi) $ holds. Therefore the $ \neg\good $ clause in the definition of $ R $ is not used at stage $ \xi $, and (M1) gives $ \base(y_\xi)\subseteq^*E $ and $ \base(y_\xi)\subseteq^*\base(y_\eta) $ for every $ \eta<\xi $.
		\end{claimproof}
		
		\begin{claim}\label{claim:vl_ramsey}
			The family $ \U $ is a Ramsey ultrafilter on $ \omega $, and $ E\in\U $.
		\end{claim}
		\begin{claimproof}{\ref{claim:vl_ramsey}}
			By Claim \ref{claim:vl_coherence}, the family $ \{\base(y_\xi):\xi<\omega_1\} $ is $ \subseteq^* $-decreasing, so $ \U $ is a proper filter. Since every $ \base(y_\xi) $ is infinite, $ \U $ is non-principal.
			
			Let $ Y\subseteq\omega $. Some $ p_\xi $ codes $ Y $. By (M2), $ \base(y_\xi)\subseteq^*Y $ or $ \base(y_\xi)\subseteq^*(\omega\setminus Y) $, so either $ Y\in\U $ or $ \omega\setminus Y\in\U $. Thus $ \U $ is an ultrafilter.
			
			Let $ c:[\omega]^2\to2 $. Some $ p_\xi $ codes $ c $, and by (M2), the set $ \base(y_\xi)\in\U $ is homogeneous for $ c $. Thus $ \U $ is Ramsey. Finally, Claim \ref{claim:vl_coherence} gives $ \base(y_\xi)\subseteq^*E $ for every $ \xi $, so $ E\in\U $.
		\end{claimproof}
		
		\begin{claim}\label{claim:vl_coanalytic}
			The family $ \idl{E} $ is coanalytic.
		\end{claim}
		\begin{claimproof}{\ref{claim:vl_coanalytic}}
			The set $ X $ is coanalytic, while $ \{y\in\ran(C):\tagg(y)=1\} $ is Borel. Hence their intersection $ \idl{E} $ is coanalytic.
		\end{claimproof}
		
		\begin{claim}\label{claim:vl_med}
			The family $ \idl{E} $ is $ \U^* $-MED.
		\end{claim}
		\begin{claimproof}{\ref{claim:vl_med}}
			We first prove $ \U^* $-eventual difference. Let $ y_\eta,y_\xi\in\idl{E} $ with $ \eta<\xi $, and suppose toward a contradiction that $ \Eq{\omega}(y_\eta,y_\xi)\in\U $.
			
			Let $ x_\xi=\real(y_\xi) $. Since $ E\in\U $ and $ E\subseteq\Eq{\omega}(x_\xi,y_\xi) $, we have $ E\cap\Eq{\omega}(y_\eta,y_\xi)\subseteq\Eq{\omega}(x_\xi,y_\eta) $, and therefore $ \Eq{\omega}(x_\xi,y_\eta)\in\U $.
			
			Since $ \tagg(y_\xi)=1 $, condition (M4) implies that $ p_\xi $ codes the real $ x_\xi $. Also $ y_\eta\in\rec(s_\xi) $. By (M3), since $ \tagg(y_\xi)=1 $, we have $ \base(y_\xi)\subseteq^*(\omega\setminus\Eq{\omega}(x_\xi,y_\eta)) $. So $ \omega\setminus\Eq{\omega}(x_\xi,y_\eta)\in\U $, contradicting $ \Eq{\omega}(x_\xi,y_\eta)\in\U $. Thus $ \idl{E} $ is $ \U^* $-e.d..
			
			Now let $ x\in\bairesp $, and choose $ \xi<\omega_1 $ such that $ p_\xi $ codes $ x $. Then $ \real(y_\xi)=x $. If $ \tagg(y_\xi)=1 $, then $ y_\xi\in\idl{E} $ and $ E\subseteq\Eq{\omega}(x,y_\xi)\in\U $. Suppose $ \tagg(y_\xi)=0 $. By (M3), there is some $ z\in\rec(s_\xi) $ such that $ \base(y_\xi)\subseteq^*\Eq{\omega}(x,z) $. Then $ z\in\idl{E} $ and $ \Eq{\omega}(x,z)\in\U $. Thus $ \idl{E} $ is $ \U^* $-MED.
		\end{claimproof}
		
		By Claims \ref{claim:vl_ramsey}, \ref{claim:vl_coanalytic}, and \ref{claim:vl_med}, it remains only to calculate the complexity of $ \U $. For every $ Y\subseteq\omega $, we have that $ Y\in\U$ if and only if $\exists y\in X\ \base(y)\subseteq^*Y $. Since $ X $ is coanalytic and $ \base(y)\subseteq^*Y $ is Borel, the ultrafilter $ \U $ is $ \bfsgm^1_2 $.
	\end{proof}

	\section{Questions}
	
	We finish with several questions suggested by the results of this paper.
	
	\begin{qst}
		For which coanalytic ideals $\mathcal{J}$ on $\omega$ containing
		$\mathrm{Fin}$ is there an analytic Van Douwen $\mathcal{J}$-MED family
		in $\omega^\omega$?
	\end{qst}
	
	Our results give a necessary condition, i.e. if such a
	family exists, then there is an infinite analytic $\mathcal{J}$-MAD family.
	
	\begin{qst}
		Which Borel ideals $\mathcal{J}$ on $\omega$ containing $\mathrm{Fin}$
		admit a Borel finitely productive $\mathcal{J}$-MED family? Which admit a
		Borel $\omega$-centered productive $\mathcal{J}$-MED family?
	\end{qst}
	
	\begin{qst}
		Are there a $P$-point $\mathcal{U}$ on $\omega$ and an analytic
		$\mathcal{U}^*$-MED family in $\omega^\omega$?
	\end{qst}

	\section*{Acknowledgments}
	
	Jialiang He was supported by the Fundamental Research Funds for the Central Universities (Grant No.~SCU2025ZLYJ-07) and the National Natural Science Foundation of China (Grant No.~12671002). During the development and preparation of this work, the authors used OpenAI ChatGPT (GPT-5.6 Sol Pro) and Google Gemini 3.1 Pro as interactive research and writing tools for literature exploration, conceptual and structural development, formulation and refinement of definitions, proof development, and manuscript drafting. All AI-assisted material was reviewed and revised by the authors, who take full responsibility for the mathematical content and the final version of the paper.

	\bibliographystyle{plain}
	\bibliography{mybibtex}
\end{document}